\documentclass[EJP]{ejpecp} 

\usepackage{accents,bm,bbm,mathrsfs}
\usepackage{graphicx,color,tikz,caption,subcaption}
\newcommand{\ubar}[1]{\underaccent{\bar}{#1}}

\newcommand{\abs}[1]{\left\lvert#1\right\rvert}

\newcommand{\setb}[2]{ \left\{ #1 \ \middle| \  #2 \right\}  }

\def\CF{{\widehat{\mathscr{P}}}}
\def\One{\mathbbm{1}} 
\def\D{{\mathcal{D}}}
\def\S{{\mathcal{S}}}

\def\R{{\mathcal{R}}}
 
\def\C{ \mathbb{C}}
\def\Z{ \mathbb{Z}}

\def\N{ \mathbb{N}}
\def\R{ \mathbb{R}}

\def\drm{\mathrm{d}}

\def\Der{\mathrm{D}}

\def\reg{\tau}
\def\rate{\rho}

\def\G{\mathcal{G}}

\def\indmom{p_{\max}}
\def\indasysup{\beta_0}

\def\indlocsup{\beta_\infty}
\def\indlocinf{\ubar{\beta}_\infty}

\SHORTTITLE{Wavelet Analysis of the Besov Regularity of L\'evy White Noise} 

\TITLE{Wavelet Analysis of the Besov Regularity of \\ L\'evy White Noise} 

\AUTHORS{%
  Shayan Aziznejad\footnote{Ecole polytechnique fédérale de Lausanne, Switzerland.}
  \and 
  Julien Fageot\footnote{Harvard University, USA.} }


\AMSSUBJ{60G51, 42C40, 46E35, 60G20} 

\SUBMITTED{March 31, 2020} 
\ACCEPTED{November 11, 2020.} 

\VOLUME{0}
\YEAR{2012}
\PAPERNUM{0}
\DOI{vVOL-PID}

\ABSTRACT{We characterize the local smoothness and the asymptotic growth rate of the L\'evy white noise.
We do so by characterizing the weighted Besov spaces in which it is located.
We extend  known results in two ways.
First, we obtain new bounds for the local smoothness via the Blumenthal-Getoor indices of the L\'evy white noise.
We also deduce the critical local smoothness when the two indices coincide, 
which is true for symmetric-$\alpha$-stable, compound Poisson, and symmetric-gamma white noises to name a few.
Second, we express the critical asymptotic growth rate in terms of the moment properties of the L\'evy white noise. 
Previous analyses only provided lower bounds for both the local smoothness and the asymptotic growth rate.
Showing the sharpness of these bounds requires us to determine in which Besov spaces a given L\'evy white noise is (almost surely) \emph{not}.
Our methods are based on the wavelet-domain characterization of Besov spaces and precise moment estimates for the wavelet coefficients of the L\'evy white noise.}

\begin{document}



\section{Introduction and Main Results} \label{sec:intro}

We study the Besov regularity of L\'evy white noises.
We are especially interested in identifying the critical local smoothness and the critical asymptotic growth rate of those random processes for any integrability parameter $p\in (0,\infty]$.
In a nutshell, our contributions are as follows.

\begin{enumerate}

	\item \emph{Wavelet Methods for L\'evy White Noises.} First appearing in the eighties, especially in the works of Y. Meyer~\cite{MeyerWaO}, I. Daubechies~\cite{Daubechies92}, and S. Mallat~\cite{Mallat1999}, wavelet techniques have become primary tools in functional analysis~\cite{Triebel2008function}.
	As such, they are a natural choice to study random processes, as is done, for instance, with fractional Brownian motion~\cite{Meyer1999wavelets}, S$\alpha$S processes~\cite{Pad2015optimality}, and with solutions of singular stochastic partial differential equations~\cite{Hairer2014theory,Hairer2017reconstruction}.
	In this paper, we demonstrate that wavelet methods are also adapted to the analysis of the L\'evy white noise.
	In particular, all our results are derived using the wavelet characterization of weighted Besov spaces.
	\vspace{0.1cm}
	
		\item \emph{New Moment Estimates for L\'evy White Noise.} The wavelet method allows us to obtain lower and upper bounds for the moments of a L\'evy white noise as a function of the wavelet scale. Our moment  estimates, which are new contributions to the rich literature on the moments of L\'evy and L\'evy-type processes~\cite{Deng2015shift,Kuhn2017existence,kuhn2017levy,Laue1980remarks,Luschgy2008moment}, are fundamental for our study of the Besov regularity of the L\'evy white noise. 
			\vspace{0.1cm}

	\item \emph{Besov Regularity of L\'evy White Noises.}
	Regularity properties are usually stated in terms of the inclusion of the process in some weighted Besov spaces (positive result).
	In order to show that such a characterization is sharp, it is  of interest to identify the smoothness spaces in which the process is \emph{not included} (negative result).
	To the best of our knowledge,  very  little is known in this direction.
	A precise answer to this question requires a more evolved analysis as compared to positive results.  
 	We achieve this goal thanks to the use of wavelets.
	
	It is worth noting that our analysis requires the identification of a new index associated to a L\'evy white noise, characterized by moment properties. 
	By relying on this index, our negative results suggest moreover that some of the previous state-of-the-art inclusions are not sharp.
	We are then able to improve some of these results, in particular for the growth properties of the L\'evy white noise.
					\vspace{0.1cm}

	\item \emph{Critical Local Smoothness and Asymptotic Rate.} The combination of positive and negative results allows us to determine the critical Besov parameters of a L\'evy white noise, both for the local smoothness and the asymptotic behavior. The results are summarized in Theorem \ref{maintheo:regdecaylimit}. Two consequences are the characterization of the critical Sobolev and H\"older-Zigmund regularities of the L\'evy white noise in Corollary \ref{coro:SHZ}.
	
\end{enumerate}

	\subsection{Local Smoothness and Asymptotic Rate of Tempered Generalized Functions} \label{sec:taurho}

We construct random processes as random elements in the space $\S'(\R^d)$ of tempered generalized functions from $\R^d$ to $\R$ (see Section \ref{subsec:noises}). 
We  therefore describe their local and asymptotic properties as we would  for a (deterministic) tempered generalized function.
To do so, we rely on the family of weighted Besov spaces. They are embedded in $\S'(\R^d)$ and allow for the  joint  study of the local smoothness and the asymptotic behavior of a generalized function. 
 
Besov spaces are denoted by $B_{p,q}^{\tau}(\R^d)$, with $\tau \in \R$ the \emph{smoothness}, $p \in (0,\infty]$ the \emph{integrability parameter}, and $q \in (0,\infty]$ a secondary parameter.
In this paper, we focus on the case $p=q$ and we use the simplified notation $B_{p,p}^{\tau}(\R^d) = B_p^{\tau}(\R^d)$ for those spaces which are also referred to as Slobodeckij spaces after~\cite{slobodeckii1958generalized}. See~\cite{di2011hitchhiker} or~\cite[Section 2.2.1]{Triebel2010theory} for more details.
We say that $f$ is in the weighted Besov space $B_p^{\reg}(\R^d;\rate)$ with \emph{weight exponent} 
$\rho \in \R$ if $\langle \cdot \rangle^{\rho} \times f$ is in the classic Besov space 
$B_p^{\reg}(\R^d)$, with the notation 
$\langle \bm{x} \rangle = (1 + \lVert \bm{x} \rVert^2)^{1/2}$.
We precisely define weighted Besov spaces in Section \ref{subsec:besovspaces} in terms of wavelet expansions.  
For the time being, it is sufficient to remember that the space of tempered generalized functions satisfies \cite[Proposition 1]{Kabanava2008tempered}
\begin{equation} \label{eq:decomposeSprime}
	\S'(\R^d) = \bigcup_{\tau, \rho \in \R} B_p^{\tau}(\R^d;\rho)
\end{equation} 
for any fixed $0 < p \leq \infty$.
Ideally,  we aim at identifying in which weighted Besov space  a given $f \in \S'(\R^d)$ is. The relation \eqref{eq:decomposeSprime} implies that, for any $p$, there exists some $\tau,\rho \in \R$ for which this is true.
For a fixed $p$, Besov spaces are continuously embedded in the sense that, for $\reg, \reg_0,  \reg_1 \in \R$ and $\rate, \rate_0 , \rate_1 \in \R$  such that $\reg_0 \geq \reg_1$ and $\rate_0 \geq \rate_1$, we have
\begin{equation} \label{eq:trivialembed}
	B_{p}^{\reg_0}(\R^d;\rate)  \subseteq B_p^{\reg_1}(\R^d;\rate)  \quad \text{and} \quad
	B_{p}^{\reg}(\R^d;\rate_0)  \subseteq B_p^{\reg}(\R^d;\rate_1).
\end{equation}
To characterize the properties of $f \in \S'(\R^d)$, the key is to determine the two critical exponents $\tau_p(f) \in (-\infty, \infty]$ and $\rho_p(f) \in (-\infty, \infty]$ such that
\begin{itemize}
	\item if $\reg < \reg_p(f)$ and $\rate < \rate_p(f)$, then $f \in B_p^\reg(\R^d;\rate)$; while
	\item  if $\reg > \reg_p(f)$ or $\rate > \rate_p(f)$, then $f \notin B_p^\reg(\R^d;\rate)$.
\end{itemize}
The case $\tau_p(f) = \infty$ corresponds to infinitely smooth functions, and \mbox{$\rho_p(f) = \infty$} means that $f$ is rapidly decaying. 
The quantity $\tau_p(f)$ measures the \emph{local smoothness} and $\rate_p(f)$ the \emph{asymptotic rate} of $f$ for the integrability parameter  $p$. When $\rate_p(f)\leq 0$ (which will be the case for the L\'evy white noise), we talk about the \emph{asymptotic growth rate} of $f$.

	\subsection{Local Smoothness and Asymptotic Growth Rate of L\'evy White Noises} \label{subsec:taurhointro}

The complete family of L\'evy white noises, defined as random elements in the space of generalized functions, was introduced by I.M. Gel'fand and N.Y. Vilenkin \cite{GelVil4}. 
Recently, R. Dalang and T. Humeau completely characterized the L\'evy white noises located in the space $\S'(\R^d)$ of tempered generalized functions \cite{Dalang2015Levy}.
We briefly  introduce in this section the concepts required to state our main results. A more complete exposition is given in Section \ref{sec:maths}.

Our main contributions concern the inclusion of a  L\'evy white noise $w$ in weighted Besov spaces. It includes positive   ($w$ is almost surely \emph{in} a given Besov space) and negative  ($w$ is almost surely \emph{not in} a given Besov space) results. 
In order to characterize the  local smoothness $\tau_p(w)$ and the asymptotic growth rate $\rate_p(w)$, let 
$X = \langle w , \One_{[0,1]^d} \rangle$ 
be the random variable that corresponds to the integration of the L\'evy white noise $w$ over the domain $[0,1]^d$. The \emph{characteristic exponent} $\Psi$ of $w$ is the logarithm of the characteristic function of $X$. More precisely, for every $\xi \in \R$, 
\begin{equation}
	\Psi(\xi) = \log \mathbb{E} \left[ \mathrm{e}^{\mathrm{i} \xi  \langle w , \One_{[0,1]^d} \rangle } \right]
 =  \log \mathbb{E} \left[ \mathrm{e}^{\mathrm{i} \xi X } \right] .
\end{equation}

	We associate to a L\'evy white noise  its \emph{Blumenthal-Getoor indices}, defined as 
\begin{align}
	\indlocsup &= \inf \setb{p>0}{\underset{\abs{\xi} \rightarrow \infty}{\lim} \frac{\abs{\Psi(\xi)}}{\abs{\xi}^p}= 0},  \label{eq:indass} \\
	\indlocinf &= \inf \setb{p>0}{\underset{\abs{\xi} \rightarrow \infty}{\lim\inf} \frac{\abs{\Psi(\xi)}}{\abs{\xi}^p}= 0}. \label{eq:indasi}	
\end{align}
The distinction is that $\indlocsup$ considers the limit, while $\indlocinf$ deals with the inferior limit. In general, one has that $0 \leq \indlocinf \leq \indlocsup \leq 2$. The Blumenthal-Getoor indices are linked to the local behavior of L\'evy processes and L\'evy white noises (see Section \ref{subsec:indices} for more details).
	In addition, we introduce the \emph{moment index} of the L\'evy white noise  $w$ as
\begin{align} 
	\indmom &= \sup \setb{p>0}{\mathbb{E}[\lvert\langle w ,\One_{[0,1]^d}\rangle \rvert^p] < \infty},	 \label{eq:indmom}
\end{align}
which is closely related---but in general not identical---to the Pruitt index (see Section \ref{subsec:indices}).
As we shall see, $\indmom \in (0, \infty]$ fully characterizes the asymptotic growth rate of $w$.
The class of L\'evy white noises is  rich, and includes Gaussian and  compound Poisson white noises. We summarize the results of this paper in Theorem \ref{maintheo:regdecaylimit}. We use the convention that $1/ p = 0$ when $p = \infty$.

\begin{theorem} \label{maintheo:regdecaylimit}
Consider a L\'evy white noise $w$ with Blumenthal-Getoor indices $0\leq \indlocinf \leq \indlocsup \leq 2$ and moment index $0 < \indmom \leq \infty$. We fix $0<p\leq \infty$.
\begin{itemize}
	\item If $w$ is a Gaussian white noise, then, almost surely,
	\begin{equation} \label{eq:GNtaurho}
	\reg_p(w) = - \frac{d}{2} \text{ and } \rate_p(w) = - \frac{d}{p}.
	\end{equation}
	\item If $w$ is a compound Poisson white noise, then, almost surely,
	\begin{equation} \label{eq:PNtaurho}
	\reg_p(w) =  \frac{d}{p} -d   \text{ and } \rate_p(w) = - \frac{d}{\min(p,\indmom)}.
	\end{equation}		
	
	\item If $w$ is a L\'evy white noise and not a Gaussian white noise, then, almost surely,
		\begin{equation} \label{eq:LNtaurho}
 \frac{d}{\max(p,\indlocsup)} -d    \leq
\reg_p(w) \leq  \frac{d}{\max(p,\indlocinf)} -d.
	\end{equation}		
	In particular, if $p \geq \indlocsup$, then $\tau_p(w) = d / p - d$. 
	
	\item If $w$ is a L\'evy white noise and not a Gaussian white noise, then, almost surely, if $p\in (0,2)$, $p$ is an even integer, or $p = \infty$,
	\begin{equation} \label{eq:LNtaubelow}
 \rate_p(w) = - \frac{d}{\min(p,\indmom)},
 	\end{equation}	
	and for any $0<p \leq \infty$, 	
	\begin{equation} \label{eq:LNtaubelow2}
 \rate_p(w) \geq - \frac{d}{\min(p,\indmom)}.
 	\end{equation}	
 					
\end{itemize}
\end{theorem}
 In a nutshell, Theorem \ref{maintheo:regdecaylimit} provides:
\begin{enumerate}
\item A full characterization of the local smoothness and the asymptotic growth rate for Gaussian and compound Poisson white noises;
\item A characterization of the asymptotic growth rate for any L\'evy white noise for integrability parameter $p\leq 2$, $p$ an even integer, or $p=\infty$; and
\item A full characterization of the local smoothness of a L\'evy white noise for which $\indlocinf=\indlocsup$; that is, for any $p \in (0,\infty]$, 
		\begin{equation} \label{eq:LNtaurhobis}
\reg_p(w) = \frac{d}{\max(p,\indlocsup)} -d.
	\end{equation}	
\end{enumerate}

We discuss the remaining cases---the local smoothness for $\indlocinf < \indlocsup$ and the asymptotic growth rate for $p>2,  p\notin 2 \N$---in Section \ref{sec:conclusiveremarks}.
Two direct consequences are the identification of the Sobolev ($p=2$) and H\"older-Zigmund ($p= \infty$) regularity of L\'evy white noises.
	
	\begin{corollary} \label{coro:SHZ}
		Let $w$ be a L\'evy white noise in $\S'(\R^d)$ with Blumenthal-Getoor indices $0\leq \indlocinf \leq \indlocsup \leq 2$ and moment index $0 < \indmom \leq \infty$. Then, the Sobolev local smoothness and asymptotic growth rate ($p=2$) are 
		\begin{equation}
			\tau_2(w) = - \frac{d}{2} \quad \text{and} \quad \rho_2(w) = - \frac{d}{\min ( 2, \indmom )}.
		\end{equation}
	Moreover, the H\"older-Zigmund local smoothness and asymptotic growth rate ($p=\infty$) are 
	\begin{align}
			\tau_\infty(w) = - \frac{d}{2} \quad &\text{and} \quad \rho_\infty(w) = 0  \quad \text{if } w \text{ is Gaussian, and} \\
			\tau_\infty(w) = -  d \quad &\text{and} \quad \rho_\infty(w) = - \frac{d}{ \indmom } \quad \text{otherwise}.
	\end{align}
	\end{corollary}
	
	\begin{proof}
	The case $p=2$ is directly deduced from Theorem \ref{maintheo:regdecaylimit} and {the relation $\indlocinf \leq \indlocsup \leq 2$. 
	For $p= \infty$, we use again Theorem \ref{maintheo:regdecaylimit} with $p=\infty$. 
	The Gaussian case follows from~\eqref{eq:GNtaurho}. 
	For the non-Gaussian case, \eqref{eq:LNtaurho} gives $\tau_\infty(w)$, while \eqref{eq:LNtaubelow} gives $\rho_\infty(w)$.}
	\end{proof}	
	
	\subsection{Local Smoothness of L\'evy Processes}

	It is worth noting that the Sobolev regularity is the same---$\tau_2(w) = - d/2$---for any L\'evy white noise.
	We also observe that the H\"older-Zigmund regularity of any non-Gaussian L\'evy white noise is $(-d)$ (which is also the regularity of a Dirac impulse), the Gaussian case being different and reaching a smoothness of $(-d/2)$. {In the one-dimensional setting ($d=1$),} this is reminiscent to the fact that the Brownian motion is the only continuous random process with independent and stationary increments, the other L\'evy processes being only \emph{c\`adl\`ag} (French acronym for functions that are right continuous with left limits at every points)~\cite{Bertoin1998levy}. Using Theorem \ref{maintheo:regdecaylimit}, we deduce the local smoothness of L\'evy processes in  Corollary \ref{coro:levyprocess}. 

	\begin{corollary} \label{coro:levyprocess}
		{Let $s : \R \rightarrow \R$ be a one-dimensional L\'evy process} with Blumenthal-Getoor indices $0 \leq \indlocinf \leq \indlocsup \leq 2$. Then, we have almost surely that, for any $0<p\leq \infty$,
	\begin{align}
			\tau_p(s) & = \frac{1}{2} \quad \text{if } s \text{ is the Wiener process, and } \\
		 	\tau_p(s) &=\frac{1}{p}  \quad  \text{if } s \text{ is a compound Poisson process}.
	\end{align}
	In the general case, we have almost surely that, for any $0<p\leq \infty$,
	\begin{equation}
	\frac{1}{\max(p,\indlocsup)}\leq \tau_p(s)  \leq \frac{1}{\max(p,\indlocinf)}.
	\end{equation}
	\end{corollary}

\begin{proof}
	A one-dimensional L\'evy white noise $w$ is the weak derivative of the corresponding L\'evy process $s$ with identical characteristic exponent. 
	This well-known fact has been rigorously shown in the sense of generalized random processes in \cite[Definition 3.4 and Proposition 3.17]{Dalang2015Levy}.
	A direct consequence is that $\tau_p(s) = \tau_p(w) + 1$, where $w = s'$. 
	Then, Corollary \ref{coro:levyprocess} is a reformulation of the local smoothness results of Theorem \ref{maintheo:regdecaylimit} with $d=1$. 
\end{proof}

	\subsection{Related Works on L\'evy Processes and L\'evy White Noises}
	\label{sec:related}
In this section, for comparison purposes, we reinterpret all the results in terms of the critical smoothness  and asymptotic growth rate of the considered random processes.

	\textit{L\'evy Processes.}
	Most of the attention has been  so far  devoted to classic L\'evy processes. 
	The Wiener process was studied in \cite{ciesielski1993orlicz,Ciesielski1993quelques,hytonen2008,Roynette1993,Sjogren1982riemann,veraar2009correlation}, while \cite{Ciesielski1993quelques} also contains results on the Besov regularity of fractional Brownian motions and S$\alpha$S processes. By exploiting the self-similarity of the stable processes, Ciesielski \emph{et al.} obtained the following results for the Gaussian \cite[Theorem IV.3]{Ciesielski1993quelques} and  stable non-Gaussian \cite[Theorem VI.1]{Ciesielski1993quelques} scenarios:
	\begin{align}
		\frac{1}{2} \leq \tau_p(s_{\mathrm{Gauss}})  \leq \frac{1}{\min(2,p)} , \label{eq:BMtau} \\
		\frac{1}{\max(p,\alpha)} \leq \tau_p(s_{\alpha} ) \leq  \frac{1}{\min(p,\alpha)}, \label{eq:SaStau}
	\end{align}	
for any $p \geq 1$, where $s_{\mathrm{Gauss}}$ is the Brownian motion and $s_{\alpha}$ is the S$\alpha$S process with parameter $1 < \alpha < 2$. 

The complete family of L\'evy processes---and, more generally, of L\'evy-type processes---has been considered by R. Schilling in a series of papers~\cite{Schilling1997Feller,Schilling1998growth,Schilling2000function} synthesized in \cite[Chapter V]{Bottcher2014levy} and by V. Herren~\cite{Herren1997levy}. 
To summarize, Schilling has shown that, for a L\'evy process $s$ with indices $0 \leq \indlocsup \leq 2$ and $0<\indmom \leq \infty$,
	\begin{align}
		\frac{1}{\max(p,\indlocsup)} &\leq \tau_{p}(s) \leq \frac{1}{p}, \label{eq:boundtauXSchilling} \\
		 -\frac{1}{p} - \frac{1}{\min(\indmom,2)} &\leq \rate_p(s). \label{eq:boundrhoXSchilling}
	\end{align}	

	 We observe that  \eqref{eq:BMtau}, \eqref{eq:SaStau}, and \eqref{eq:boundtauXSchilling}  are consistent with Corollary \ref{coro:levyprocess}. 
	Moreover, our results provide an improvement by showing that the lower bounds of  \eqref{eq:BMtau} and \eqref{eq:SaStau} are actually sharp.
		Finally, we significantly improve the upper bound of \eqref{eq:boundtauXSchilling} for general L\'evy processes. 

 In contrast to the smoothness, the growth rate \eqref{eq:boundrhoXSchilling} of the L\'evy process $s$  does not seem to be related to the one of  its derivative the L\'evy white noise $w = s'$   by a simple relation. In particular, the rate of $s$ is expressed in terms of the Pruitt index  $ \indasysup = \min ( \indmom , 2)$, conversely to $\indmom$ for $w$ (see Section \ref{subsec:indices}). 
		This needs to be confirmed by a precise estimation of $\rho_p(s)$ for which only a lower bound is known. 
		
		\textit{L\'evy White Noises.}
M. Veraar extensively studied  the local Besov regularity of the $d$-dimensional Gaussian white noise. As a corollary of  \cite[Theorem 3.4]{Veraar2010regularity}, one then deduces that $\tau_{p}(w_{\mathrm{Gauss}}) = - d /2$. 
This work is based on the specific properties of the Fourier series expansion of the random process under the Gaussianity assumption, and cannot be directly adapted to L\'evy white noises.

In our own works, we have investigated the question for general L\'evy white noises in dimension $d$ in the periodic \cite{Fageot2017besov} and global settings \cite{Fageot2017multidimensional}. We obtained the lower bounds
	\begin{align}
		\frac{d}{\max(p,\indlocsup)} - d \leq &\tau_{p}(w), \quad  \text{ and }\label{eq:boundtauwus} \\
		- \frac{d}{\min(p,\indmom,2)} \leq &\rate_p(w). \label{eq:boundrhowus}
	\end{align}	
	These estimates are improved by	Theorem \ref{maintheo:regdecaylimit}, which  provides an upper bound for $\tau_{p}(w)$ and shows that \eqref{eq:boundtauwus} is sharp when $\indlocsup = \indlocinf$. 
	It is also worth noticing that the lower bound of \eqref{eq:boundrhowus} is sharp if and only if $\indmom \leq 2$.

	\subsection{Sketch of Proof and the Role of Wavelet Methods} \label{subsec:sketch}
	
	Our techniques are based on the wavelet characterization of Besov spaces, as presented by H. Triebel in~\cite{Triebel2008function}. We shall see that wavelets are especially relevant to the analysis of L\'evy white noises.
	
	We  briefly present the strategy of the proof of Theorem \ref{maintheo:regdecaylimit} when the ambiant dimension is $d=1$.  	
	The  general case $d\geq 1$ is analogous and will be comprehensively addressed  in the rest of the paper.
	Let $(\psi_M,\psi_F)$ be the (mother, father) Daubechies wavelets of a fixed order (the choice of the order has no influence on the results as soon as it is large enough). 
	For $j \in \N$ and $k\in \Z$, we define the rescaled and shifted functions $\psi_{F,k} = \psi_F (\cdot - k)$ and $\psi_{j,M,k} = 2^{j/2} \psi_M (2^j \cdot - k)$.
	Then, the family $(\psi_{F,k})_{k \in \Z} \cup (\psi_{j , M, k})_{j \in \N , k\in \Z}$ forms an orthonormal basis of $L_2(\R)$~\cite{Daubechies92}.  
	For a given one-dimensional L\'evy white noise $w$, one considers the family of random variables 
	\begin{equation} \label{eq:waveletrandomvariable}
	(\langle w , \psi_{F,k}\rangle )_{k\in \Z} \cup (\langle w , \psi_{j,M,k} \rangle)_{j \in \N , k\in \Z}.
	\end{equation}
	We then have that $w = \sum_{k\in \Z} \langle w ,\psi_{F,k}\rangle \psi_{F,k}+\sum_{j \in \N}\sum_{k\in \Z} \langle w ,\psi_{j,M,k}\rangle \psi_{j,M,k}$, where the convergence is almost sure in $\S'(\R)$.
	
	Then, for $0<p < \infty$  (the case $p = \infty$ will be deduced by embedding and is not discussed in this section) and $\reg, \rate\in\R$, the random variable 
\begin{equation} \label{eq:methodsbesov}
	\lVert w \rVert_{B_p^\tau(\R ;\rho)} =  \left(
		\sum_{k\in \Z}  \langle k \rangle^{\rate p}  \lvert \langle w , \psi_{F,k} \rangle \rvert^p
		+
		\sum_{j \in \N} 2^{j(\reg p -1 + \frac{p}{2})} 
		\sum_{k\in \Z} \langle 2^{-j} k \rangle^{\rate p}    \lvert \langle w ,\psi_{j,M,k}\rangle \rvert^p \right)^{1/p}
\end{equation}
is well-defined  and takes values in $[0,\infty]$.  
Here, $\lVert w \rVert_{B_p^\tau(\R ;\rho)}$ is the Besov (quasi-)norm of the L\'evy white noise (see Section \ref{subsec:besovspaces}). This means that $w$ is a.s. (almost surely) in $B_p^\tau(\R ;\rho)$ if and only if $\lVert w \rVert_{B_p^\tau(\R ;\rho)} < \infty$  a.s., and  a.s.  not in  $B_p^\tau(\R ;\rho)$ if and only if $\lVert w \rVert_{B_p^\tau(\R ;\rho)}  = \infty$ a.s.

We then fix $0<p < \infty$. We assume that we have guessed the values $\reg_p(w)$ and $\rate_p(w)$  introduced in Section \ref{sec:intro}. Here are the main steps leading to the proof that these values are effectively the critical ones.

\begin{itemize}
	
	\item For $ \tau < \tau_p(w) $ and $\rho < \rho_p(w)$, we show that $\lVert w \rVert_{B_p^\tau(\R;\rho)} < \infty$ a.s. For $p < \indmom$ (see \eqref{eq:indmom}), we establish the stronger result  $\mathbb{E} \left[ \lVert w \rVert_{B_p^\tau(\R;\rho)}^p \right] < \infty$. This requires moment estimates for the wavelet coefficients of a L\'evy white noise, which gives  a precise estimation of the behavior of $\mathbb{E}[  \lvert \langle w , \psi_{j,M, k} \rangle \rvert^p ]$ as $j$ goes to infinity. When $p > \indmom$,  the random variables $\langle w , \psi_{j,M, k} \rangle$ have an infinite $p$th moment and the present method is not applicable. In that case, we actually deduce the result using   embedding relations between Besov spaces. It turns out that this approach is sufficient to obtain sharp   results. 
	
	\item For $\tau > \tau_p(w)$, we show that $\lVert w \rVert_{B_p^\tau(\R;\rho)} = \infty$ a.s.
		To do so, we only consider the mother wavelet and truncate the sum over $k$ to yield the lower bound
		\begin{equation} \label{eq:principlesmoothness}
		\lVert w \rVert_{B_p^\tau(\R;\rho)}^p \geq 
		 C \sum_{j \in \N} 2^{j(\reg p - 1 + \frac{p}{2})} 
		\sum_{0\leq k < 2^{j}}  \lvert \langle w , \psi_{j,M, k} \rangle \rvert^p
		\end{equation}
		for some constant $C$ such that   $\langle 2^{-j} k  \rangle^{\rate p} \geq C$ for every $j \in \N$ and $0\leq k < 2^{j}$.
 		We then need to show  that the wavelet coefficients $\langle w , \psi_{j,M, k} \rangle$ cannot be too small altogether using Borel-Cantelli-type arguments. Typically, this requires us to control the evolution  of quantities such as $\mathbb{P}(\abs{\langle w ,\psi_{j,M,k} \rangle } > x)$ with respect to $j$ and is again based on moment estimates.  
		
	\item For $\rho > \rho_p(w)$, we show again that $\lVert w \rVert_{B_p^\tau(\R;\rho)} = \infty$ a.s.
		This time, we only consider the father wavelet  in \eqref{eq:methodsbesov}  and use the lower bound 
\begin{equation}  \label{eq:principledecay}
	\lVert w \rVert_{B_p^\tau(\R;\rho)}^p  \geq	\sum_{k\in \Z} \langle  k  \rangle^{\rate p} \lvert \langle w , \psi_{F,k} \rangle \rvert^p.
\end{equation}			
	A Borel-Cantelli-type argument is again used to show that the $ \lvert \langle w , \psi_{F,k} \rangle \rvert $ cannot be too small altogether, and that the Besov norm is a.s.  infinite.
\end{itemize}

	The rest of the paper is dedicated to the proof of Theorem \ref{maintheo:regdecaylimit}. The required  mathematical concepts---L\'evy white noises as generalized random processes and weighted Besov spaces---are laid out in Section \ref{sec:maths}.
	In Sections \ref{sec:Gaussian}, \ref{sec:Poisson}, and \ref{sec:finitemoments}, we consider  the case of Gaussian white noises, compound Poisson white noises, and finite-moments L\'evy white noises, respectively. { Section~\ref{sec:estimates} provides some new moment estimates for Lévy white noises that are preparatory to the upcoming sections. 
	The general case is deduced in Section \ref{sec:general}, where we provide the proof of Theorem \ref{maintheo:regdecaylimit}. } Finally, we discuss our results and give important examples in Section \ref{sec:blabla}. 
	
\section{Preliminaries: L\'evy White Noises and Weighted Besov Spaces} \label{sec:maths}

	\subsection{L\'evy White Noises as Generalized Random Processes} \label{subsec:noises}
	
	The theory of generalized random processes was initiated independently by K. It\^o~\cite{Ito1954distributions} and  I.M. Gel'fand~\cite{Gelfand1955generalized} in the 50's and corresponds to the probabilistic counterpart of the theory of generalized functions of L. Schwartz. It was later brought to light by Gel'fand himself together with N.Y. Vilenkin in~\cite[Chapter III]{GelVil4}. 
	In this framework, a generalized random process is characterized by its effects against test functions. This allows to consider random processes that are not necessarily defined pointwise, as is the case for the L\'evy white noise. The theory of generalized random processes,  besides being very general, appears to be very flexible for the construction and analysis of random processes. It is a powerful alternative to more classic approaches, as argumented in~\cite{Cartier1963processus,Fernique1967processus}. The theory of generalized random processes is used as the ground for generalized CARMA processes~\cite{Brockwell2010carma} and fields~\cite{berger2019levydriven2,berger2020levy}, for conformal field theory in statistical physics~\cite{Abdesselam2020second}, for studying the solutions of stochastic differential PDEs~\cite{dalang2019random,Ito1984foundations,walsh1986introduction}, and as random models in signal processing~\cite{Bostan2013map,clarkson2016characteristic,Fageot2015wavelet,Unser2014sparse}.

	We shall define random processes as random elements of the space $\S'(\R^d)$, that we introduce now.  
	Let	$\S(\R^d)$ be the space of rapidly decaying smooth functions from $\R^d$ to $\R$. It is endowed with its natural Fr\'echet nuclear topology \cite{Treves1967}. Its topological dual is the space of tempered generalized functions $\S'(\R^d)$.
	It is endowed with the strong topology and $\mathcal{B}(\S'(\R^d))$ denotes the Borelian $\sigma$-field for this topology. 
	Note that $\S'(\R^d)$ can be endowed with other natural $\sigma$-fields: the one associated to the weak-* topology or the cylindrical $\sigma$-field generated by the cylinders 
	$$\{ u \in \S'(\R^d), \ (\langle u ,\varphi_1 \rangle,\ldots, \langle u ,\varphi_N \rangle ) \in B \}$$
	 for $N\geq 1$, $\varphi_n \in \S(\R^d)$, and $B$ a Borelian subset of $\R^N$. However, these different $\sigma$-fields are known to coincide in this case~\cite[Proposition 3.8 and Corollary 3.9]{Bierme2017generalized}\footnote{This is true in general for the dual of a nuclear Fr\'echet space. Note that this is not obvious and is typically not true for other spaces of generalized functions, such as $\mathcal{D}'(\R^d)$~\cite{Ito1984foundations}.}. See also It\^o's~\cite{Ito1984foundations} and Fernique's monographs \cite{Fernique1967processus} for general discussions on the measurable structures of function spaces.
	Throughout the paper, we fix a complete probability space $(\Omega, \mathcal{F}, \mathscr{P})$.
	
	\begin{definition} \label{def:GRP}
		A measurable function $s$ from $(\Omega, \mathcal{F})$ to $(\S'(\R^d), \mathcal{B}(\S'(\R^d)))$ is called a \emph{generalized random process}. Its \emph{probability law} is the probability measure on $\S'(\R^d)$ defined for $B \in \mathcal{B}(\S'(\R^d))$ by 
		\begin{equation}
			\mathscr{P}_s ( B ) = \mathscr{P} (\{ \omega \in \Omega , \ s(\omega) \in B \}).
		\end{equation}
		The \emph{characteristic functional} of $s$ is the functional $\CF_s : \S(\R^d) \rightarrow \C$ such that 
		\begin{equation}
			\CF_s(\varphi) = \int_{\S'(\R^d)} \mathrm{e}^{\mathrm{i} \langle u , \varphi \rangle } \drm \mathscr{P}_s (u).
		\end{equation}
	\end{definition}
	
It turns out that the characteristic functional is continuous, positive-definite over $\S(\R^d)$, and normalized such that $\CF_s(0) = 1$. The converse of this result is also true: if $\CF$ is a continuous and positive-definite functional over $\S(\R^d)$ such that $\CF(0) = 1$, then  it is the characteristic functional of a generalized random process in $\S'(\R^d)$. This  is known as the Bochner-Minlos theorem. It was initially proved in~\cite{Minlos1959generalized} and uses the nuclearity of $\S'(\R^d)$. See \cite[Theorem 2.3]{simon1979functional} for an elegant proof based on the Hermite expansion of tempered generalized functions~\cite{Simon2003distributions}.
It means in particular that one can define generalized random processes via the specification of their characteristic functional. Following Gel'fand and Vilenkin, we use this principle to introduce L\'evy white noises.

We consider functionals of the form $\CF(\varphi) = \exp \left( \int_{\R^d} \Psi ( \varphi ( \bm{x} ) ) \drm \bm{x} \right)$.
It is known that $\CF$ is a characteristic functional over the space $\D(\R^d)$  of compactly supported smooth functions if and only if the function $\Psi : \R \rightarrow \C$ is continuous, conditionally positive-definite, with $\Psi(0) = 0$ \cite[Section
III-4, Theorems 3 and 4]{GelVil4}. A function $\Psi$ that satisfies  these conditions is called a \emph{characteristic exponent} and can be decomposed according to the L\'evy-Khintchine theorem \cite[Theorem 8.1]{Sato1994levy} as 
		\begin{equation} \label{eq:LK}
		 	\Psi(\xi) = \mathrm{i} \mu \xi - \frac{\sigma^2 \xi^2}{2} + \int_{\R} (\mathrm{e}^{\mathrm{i} \xi t} - 1 - \mathrm{i} \xi t \One_{\abs{t} \leq 1} ) \drm \nu (t),
		\end{equation}
		where $\mu \in \R$, $\sigma^2 \geq 0$, and $\nu$ is a \emph{L\'evy measure}, which means a positive measure on $\R$ such that $\nu (\{0\} )  = 0$ and $\int_{\R}  \inf( 1 , t^2 ) \drm \nu ( t) < \infty$. The triplet $(\mu,\sigma^2,\nu)$ is unique and called the \emph{L\'evy triplet} of $\Psi$. 

In our case, we are only interested in the definition of L\'evy white noises over $\S'(\R^d)$. This requires an adaptation of the construction of Gel'fand and Vilenkin. 
We say that the characteristic exponent $\Psi$ satisfies the \emph{$\epsilon$-condition} if there exists some $\epsilon>0$ such that $\int_{\R}  \inf( \abs{t}^\epsilon , t^2 ) \drm \nu ( t) < \infty$, with $\nu$ the L\'evy measure of $\Psi$. 
Then, the functional $\CF(\varphi) = \exp \left( \int_{\R^d} \Psi ( \varphi ( \bm{x} ) ) \drm \bm{x} \right)$ is a  characteristic functional over $\S(\R^d)$ if and only if $\Psi$ is a characteristic exponent that satisfies the $\epsilon$-condition. The sufficiency is proved in \cite[Theorem 3]{Fageot2014} and the necessity in \cite[Theorem 3.13]{Dalang2015Levy}. 

\begin{definition} \label{def:levynoise}
	A \emph{L\'evy white noise} in $\S'(\R^d)$ (or simply a L\'evy white noise) is a generalized random process $w$ with characteristic functional of the form
	\begin{equation}
		\CF_w(\varphi) = \exp\left(  \int_{\R^d} \Psi( \varphi(\bm{x} ) ) \drm \bm{x}\right) 
	\end{equation}
	for every $\varphi \in \S(\R^d)$, where $\Psi$ is a characteristic exponent that satisfies the $\epsilon$-condition.
	
	The L\'evy triplet of $w$ is denoted by $(\mu,\sigma^2, \nu)$. Then, we say that $w$ is a \emph{Gaussian white noise} if $\nu = 0$, a \emph{compound Poisson white noise} if $\mu = \sigma^2 = 0$ and $\nu = \lambda P$, with $\lambda > 0$ and $P$ a probability measure on $\R$ such that $P( \{0\} )= 0$, and a  \emph{L\'evy white noise with finite moments} if $\mathbb{E}[\abs{\langle w, \varphi \rangle}^p] < \infty$ for any $\varphi \in \S(\R^d)$ and $p>0$.
\end{definition}

 The $\epsilon$-condition is extremely mild. L\'evy white noises in $\S'(\R^d)$ include stable white noises, symmetric-gamma white noises, and compound Poisson white noises whose jumps probability measure $P$ admits a finite moment ($\int_{\R} |t|^\epsilon P(\mathrm{d} t) < \infty$ for some $\epsilon>0$)~\cite[Section 2.1.3]{Fageotthese}. 
L\'evy white noises are stationary and independent at every point, meaning that $\langle w , \varphi_1 \rangle$ and $\langle w , \varphi_2 \rangle$  are independent as soon as $\varphi_1$ and $\varphi_2 \in \S(\R^d)$ have disjoint supports~\cite[Section III-4, Theorem 6]{GelVil4}.

	One can extend the space of test functions a given L\'evy white noise can be applied to. This is done by approximating a test function $\varphi$ with functions in $\S(\R^d)$ and showing that the underlying sequence of random variables converges in probability to a random variable that we denote by $\langle w , \varphi \rangle$. This principle is developed with more generality in \cite{Fageot2017unified} by connecting the theory of generalized random processes to independently scattered random measures in the sense of B.S. Rajput and J. Rosinski \cite{Rajput1989spectral}; see also \cite{griffiths2019modelling}.
	In particular, as soon as $\varphi \in L_2(\R^d)$ is compactly supported, the random variable $\langle w, \varphi \rangle$ is well-defined \cite[Proposition 5.10]{Fageot2017unified}. Daubechies wavelets or indicator functions of measurable sets with finite Lebesgue measures satisfy this condition. This was implicitely used in Section \ref{subsec:taurhointro} when considering the random variable $\langle w,  \One_{[0,1]^d}\rangle$ and in Section~\ref{subsec:sketch} when considering the wavelet coefficients of the L\'evy white noise.
	
    { \textit{Remark.}  The random variable $\langle w , \varphi \rangle$ can be interpreted as a stochastic integral with respect to a Lévy sheets $s : \R^d \rightarrow \R$ such that $\Der_1 \ldots \Der_d \{s\} = w$, where $\Der_i$ is the partial derivative  along direction $1\leq i\leq d$. 
   We recall that Lévy sheets are multivariate generalizations of the Lévy processes~\cite{Dalang2015Levy,Dalang1992,griffiths2019modelling}. 
    In that case, we have the formal relation $\langle w , \varphi \rangle =  \int_{\R^d} \varphi(\bm{x}) \mathrm{d}s(\bm{x})$, whose precise meaning has been investigated in~\cite{Dalang2015Levy,Fageot2017unified}.}

	\subsection{The L\'evy-It\^o Decomposition of L\'evy White Noises} \label{subsec:LIdecompo}

The L\'evy-It\^o decomposition is a fundamental result of the theory of L\'evy processes. It reveals that a L\'evy process $s = (s(t))_{t \in \R}$ can be decomposed as $s = s_1 + s_2 + s_3 $, where $s_1$ is a Wiener process, $s_2$ is a compound Poisson process, and $s_3$ is a square integrable pure jump martingale, which corresponds to the small jumps of $s$~\cite[Theorem 2.4.16]{Applebaum2009levy}, \cite[Theorems 19.2 and 19.3]{Sato1994levy}. {The extension of the L\'evy-It\^o decomposition  to the multivariate setting requires to define L\'evy fields, for which different constructions are possible~\cite{Dalang1992,Durand2012multifractal,mori1992representation}. This includes L\'evy sheets, that we already mentioned and for which the L\'evy-It\^o decomposition has been extended for L\'evy sheets in~\cite[Theorem 4.6]{Adler1983representations}.
Using the Lévy-Itô decomposition of Lévy sheets, we are able to provide an identical result for the Lévy white noise. This is based on the connection between Lévy sheets and Lévy white noises, which is one of the main contribution of~\cite{Dalang2015Levy}. Indeed, the L\'evy white noise $w$ satisfies the relation
\begin{equation}
\label{eq:noiseandprocess}
    \mathrm{D}_1 \ldots \mathrm{D}_d  \{s\}  = w
\end{equation}
	for some L\'evy sheet $s$ in $\S'(\R^d)$~\cite[Defintion 3.4 and Proposition 3.17]{Dalang2015Levy}. 
    In dimension $d=1$, we recover that the (weak) derivative of the Lévy process is a Lévy white noise.}

	
	\begin{proposition}
\label{prop:LIdecompo}
	A L\'evy white noise $w$ can be decomposed as 
	\begin{equation} \label{eq:LIde}
		w = w_{1} +  w_2 + w_3
	\end{equation}		
with  $w_{1}$  a Gaussian white noise, $w_2$ a compound Poisson white noise, and $w_3$ a L\'evy white noise with finite moments, the three being independent.	
\end{proposition}

\begin{proof}
	According to \eqref{eq:noiseandprocess}, $w =\mathrm{D}_1 \ldots \mathrm{D}_d  \{s\}$ for some Lévy sheet $s :\R^d \rightarrow \R$. 
	Then, according to~\cite[Theorem 4.6]{Adler1983representations}, $s$ can be decomposed as $s = s_1 + s_2 + s_3$ where $s_1$ is a Brownian sheet, $s_2$ is a compound Poisson sheet, and $s_3$ is L\'evy sheet which is a square integrable pure jump martingale ($s_3$ corresponds to the small jumps of $s$).
	Moreover,  the three random fields $s_1,s_2,s_3$ are independent from each other. Then, the jumps of $s_3$ are bounded by construction, implying that it has finite moments  \cite[Theorem 2.4.7]{Applebaum2009levy}.
	Finally, we have that 
	\begin{equation}
	w = \mathrm{D}_1 \ldots \mathrm{D}_d  \{s\}  = \mathrm{D}_1 \ldots \mathrm{D}_d  \{s_1\}  + \mathrm{D}_1 \ldots \mathrm{D}_d  \{s_2\}   + \mathrm{D}_1 \ldots \mathrm{D}_d  \{s_3\}   := w_1 + w_2 + w_3,
	\end{equation}
	where $w_1$ is a Gaussian white noise, $w_2$ is a compound Poisson white noise, and {$w_3$ is a L\'evy white noise with finite moments. 
	This last point is indeed ensured by the fact that the Lévy measure $\nu_3$ associated to $s_3$ and therefore $w_3$ has a compact support. Hence, we have that $\int_{\R} |t|^p \mathrm{d}\nu_3 (t) < \infty$ for any $p>0$. This implies that $\mathbb{E} [ |\langle w_3 , \varphi \rangle|^p] < \infty$ for any $\varphi \in \mathcal{S}(\R^d)$ and $p > 0$ according to~\cite[Theorem 25.3]{Sato1994levy} (see also Proposition~\ref{prop:momentnoisepmax} thereafter). }
	Note moreover that $w_1$, $w_2$, and $w_3$ are independent, because the corresponding L\'evy sheets are.
\end{proof}
	
	\subsection{Indices of L\'evy White Noises} \label{subsec:indices}
	
	 { We introduce various indices associated to Lévy white noises. First of all, we exclude Lévy white noises with dominant drift via the following classic notion that appears for instance in~\cite{Deng2015shift,Schilling1998growth}. }
	
	\begin{definition}
    {  We  say that a L\'evy white noise $w$ with  characteristic exponent $\Psi$ satisfies the \emph{sector condition} if there exists $M >0$ such that
	\begin{equation} \label{eq:sectorcondition}
	\forall \xi \in \R, \quad	\abs{\Im\{ \Psi(\xi) \}  } \leq M \abs{\Re \{\Psi(\xi)\}}.
	\end{equation}}
	\end{definition}
 { This condition ensures that no drift is dominating the L\'evy white noise. For instance, the deterministic Lévy white noise $w = \mu \neq 0$ a.s., which corresponds to the Lévy triplet $(\mu,0,0)$, is such that $\Psi(\xi) = \mathrm{i} \mu \xi$ and does not satisfy the sector condition. This is also the case for $w = \mu + w_\alpha$ where $w_\alpha$ is a S$\alpha$S process with $\alpha \in (0,2]$. It is worth noting that the characteristic exponent of a symmetric L\'evy white noise is real, and therefore satisfies the sector condition.  
    In the rest of the paper, we will always assume that the sector condition is satisfied without further mention.}

	In Theorem \ref{maintheo:regdecaylimit}, the smoothness and growth rate of L\'evy white noises is characterized in terms of the indices \eqref{eq:indass}, \eqref{eq:indasi}, and \eqref{eq:indmom}. We give here some additional insight about these quantities. 
The index $\indlocsup$ was introduced by R. Blumenthal and R. Getoor \cite{Blumenthal1961sample} to characterize the behavior of L\'evy processes at the origin. This quantity appears to be related to many local properties of random processes driven by L\'evy white noises, including the Hausdorff dimension of the image set \cite{Bottcher2014levy}, the spectrum of singularities \cite{Durand2012multifractal,Jaffard1999multifractal}, the Besov regularity \cite{Bottcher2014levy,Fageot2017multidimensional,Schilling1997Feller,Schilling2000function} and more generally sample path properties~\cite{chong2019path,kuhn2019domain,Rosenbaum2009first}, the local self-similarity \cite{fageot2019scaling}, or the local compressiblity \cite{Fageot2017nterm}. Finally, the   index  $\indlocinf$   plays a crucial role in the specification of negative results, such as the identification of  the Besov spaces in which the L\'evy white noises are not. It satisfies moreover the relation $0 \leq \indlocinf \leq \indlocsup \leq 2$.

{In \cite{Pruitt1981growth}, W. Pruitt proposed the index
\begin{equation}
\indasysup = \sup \setb{p>0}{\underset{\abs{\xi} \rightarrow 0}{\lim} \frac{\abs{\Psi(\xi)}}{\abs{\xi}^p}= 0}   \label{eq:indpruitsup}
\end{equation}
 as the asymptotic counterpart of $\indlocsup$. This quantity appears in the asymptotic growth rate of the supremum of L\'evy(-type) processes \cite{Schilling1998growth} and the asymptotic self-similarity of random processes driven by L\'evy white noises \cite{fageot2019scaling}. The Pruitt index differs from the index $\indmom$ that appears  in Theorem \ref{maintheo:regdecaylimit}. Actually, the two quantities are linked by the relation $\indasysup = \min(\indmom, 2)$. This is shown by linking $\indasysup$ to the L\'evy measure~\cite{Bottcher2014levy} and knowing that $\indasysup \leq 2$ (see the appendix of \cite{Deng2015shift} for a short and elegant proof). This means that $\indasysup < \indmom$ when the L\'evy white noise has some finite $p$th moments fo $p>2$, and one cannot recover $\indmom$ from $\indasysup$ in this case. It is therefore necessary to introduce the index $\indmom$  in addition to the Pruitt index in our analysis. Note moreover that the moment index fully characterizes the moment properties of the L\'evy white noise in the following sense.}
 

\begin{proposition} \label{prop:momentnoisepmax}
Let $0 <p < \infty$ and $w$ be a L\'evy white noise with moment index $\indmom \in (0,\infty]$. We also fix a compactly supported and bounded test function $\varphi \neq 0$. 
If $p < \indmom$, then
\begin{equation}  \label{eq:quandcestfini}
	\mathbb{E} [ |\langle w , \varphi \rangle |^p ] < \infty,
\end{equation}
while if $p > \indmom$, then
\begin{equation} \label{eq:quandcestinfini}
	\mathbb{E} [ |\langle w , \varphi \rangle |^p ] = \infty.
\end{equation}
\end{proposition}

Proposition \ref{prop:momentnoisepmax} can be deduced from more general results presented in \cite{Rajput1989spectral} and \cite{Fageot2017unified}, where the set of test functions $\varphi$ such that $	\mathbb{E} [ |\langle w , \varphi \rangle |^p ] < \infty$ is fully characterized. For us, it is enough to know that the result is true for compactly supported bounded test functions, which includes Daubechies wavelets.
We provide a proof thereafter for the sake of completeness, since this result is not exactly stated as such in the literature and known results require to introduce  tools that are unnecessary for this paper. The first part \eqref{eq:quandcestfini} allows one to consider the moments of $\langle w , \varphi \rangle$ for any $p< \indmom$. The second part \eqref{eq:quandcestinfini} shows that some moments are infinite and will appear to be useful later on. 

\begin{proof}[Proof of Proposition \ref{prop:momentnoisepmax}]
The proof relies on the link between the moments of $w$ and the moments of its L\'evy measure $\nu$. 
{According to \cite[Theorem 25.3]{Sato1994levy}, a random variable $X$ with Lévy measure $\nu$  is such that 
\begin{equation}
    \mathbb{E}[ |X|^p ] < \infty \ \Longleftrightarrow \ \int_{|t| > 1} |t|^p \mathrm{d}\nu(t) < \infty.
\end{equation}
Applying this to $X = \langle w , \One_{[0,1)^d} \rangle$ (whose Lévy measure is indeed $\nu)$) and using the definition of $\indmom$ in \eqref{eq:indmom}, we deduce that
\begin{equation} \label{eq:controlemoica}
	\int_{|t| > 1} |t|^p \mathrm{d}\nu(t) < \infty \
	\text{ if } p< \indmom, \text{ and } \int_{|t| > 1} |t|^p \mathrm{d}\nu(t) = \infty \text{ if } p> \indmom.
\end{equation}}
According to \cite[Proposition 3.14]{Fageot2017unified}, we have that  $\mathbb{E} [ |\langle w , \varphi \rangle |^p ] < \infty$ if and only if
\begin{equation}
	\int_{\R^d} \int_{\R} \lvert t \varphi( \bm{x} )\rvert^p \One_{\lvert t \varphi(\bm{x} ) \rvert > 1} \mathrm{d} \nu(t) \mathrm{d}\bm{x} < \infty.
\end{equation}
Let $K$ be the compact support of $\varphi$. Moreover, the test function being non identically zero, there exists $m>0$ such that $ \mathrm{Leb}(\{ |\varphi| \geq m\}) > 0$
where $\mathrm{Leb}$ is the Lebesgue measure. 
Set $p < \indmom$, then for every $t\in \R$ and every $\bm{x}\in K$, we have that $\One_{|t \varphi(\bm{x})| > 1} \leq  \One_{|t |  \lVert \varphi\rVert_\infty > 1}$. Therefore, according to the left side of \eqref{eq:controlemoica},
\begin{align}
	\int_{\R^d} \int_{\R} \lvert t \varphi( \bm{x} )\rvert^p \One_{\lvert t \varphi(\bm{x} ) \rvert > 1} \mathrm{d} \nu(t) \mathrm{d}\bm{x}
	&\leq \int_K \int_{\R}  \lvert t  \rvert^p \lVert \varphi \rVert_\infty^p    \One_{|t |  \lVert \varphi\rVert_\infty > 1}  \mathrm{d} \nu(t) \mathrm{d}\bm{x}  \nonumber  \\
	&= \mathrm{Leb}(K) \lVert \varphi \rVert_\infty^p  \int_{|t|  > 1/\lVert \varphi \rVert_\infty} |t|^p \mathrm{d} \nu(t) < \infty,
	\end{align}	
	 proving \eqref{eq:quandcestfini}. 
	Moreover, if $p > \indmom$, then, using the right side of \eqref{eq:controlemoica} and the inequality $\One_{|t \varphi(\bm{x})| > 1} \geq \One_{|t| m > 1}$
	for every $\bm{x}\in  \{ \lvert \varphi \rvert \geq m \}$,
	we have
\begin{align}
	\int_{\R^d} \int_{\R} \lvert t \varphi( \bm{x} )\rvert^p \One_{\lvert t \varphi(\bm{x} ) \rvert > 1} \mathrm{d} \nu(t) \mathrm{d}\bm{x}
	&\geq \int_{|\varphi| \geq m} \int_{\R}  \lvert t  \rvert^p m^p    \One_{|t | m > 1}  \mathrm{d} \nu(t) \mathrm{d}\bm{x}  \nonumber \\
	&= \mathrm{Leb}(\{ |\varphi| \geq m\})m^p  \int_{|t|  > 1/m} |t|^p \mathrm{d} \nu(t) = \infty,
	\end{align}	
and \eqref{eq:quandcestinfini} is proved. 
\end{proof}

 We now summarize how the indices of the L\'evy-It\^o decomposition of a L\'evy white noise behave in Proposition \ref{prop:decomposition}.

\begin{proposition} \label{prop:decomposition}
	Let $w$ be a L\'evy white noise $w$ and let 
$w  = w_{1} +  w_2 + w_3$ 
	be its L\'evy-It\^o decomposition according to \eqref{eq:LIde} in Proposition \ref{prop:LIdecompo}, where $w_1$ is Gaussian, $w_2$ is compound Poisson, and $w_3$ have finite moments. 
	
{(i) If $w_1$, $w_2$, and $w_3$ are nonzero, then
	\begin{equation} \label{eq:firststuff}
	 \indlocinf (w_1) = \indlocsup (w_1) =  2, \quad \indlocinf (w_2)  = \indlocsup (w_2)  = 0, \quad \text{and} \quad   \indmom(w_1) = \indmom (w_3) = \infty. 
	\end{equation}
	
	(ii) If $w = w_2 + w_3$ has no Gaussian part ($w_1 = 0$) with $w_3$ nonzero, then
	\begin{align} \label{eq:secondstuff}
	\indlocinf  (w) = \indlocinf (w_3), \quad   \indlocsup (w) = \indlocsup  (w_3), \quad \text{and} \quad  \indmom (w)  = \indmom (w_2).
	\end{align}
	
	(iii) If $w = w_1 + w_2 + w_3$ with $w_1$ and $w_3$ non zero, then  
	\begin{align}\label{eq:thirdstuff}
	\indlocinf  (w) =    \indlocsup (w) = 2 \quad \text{and} \quad \indmom (w) = \indmom (w_2).
	\end{align}	}
	\end{proposition}

\begin{proof}
Let $\Psi$ be the characteristic exponent of $w$, and $(\mu,\sigma^2,\nu)$ be its L\'evy triplet. We assume that $\mu = 0$, what has no impact on the indices.
The L\'evy-It\^o decomposition corresponds to the following sum for the characteristic exponent:
	\begin{align}
		\Psi(\xi) 	 = \underbrace{ - \frac{\sigma^2\xi^2}{2}}_{\Psi_1(\xi)}  + \underbrace{\int_{\R} (\mathrm{e}^{\mathrm{i} \xi t} - 1 )  \One_{\abs{t} > 1} \drm \nu (t)  }_{\Psi_2(\xi)}+  \underbrace{  \int_{\R} (\mathrm{e}^{\mathrm{i} \xi t} - 1 - \mathrm{i} \xi t  ) \One_{\abs{t} \leq 1} \drm \nu (t)}
_{\Psi_3(\xi)},
	\end{align}
	where, $\Psi_{1}$, $\Psi_{2}$, and $\Psi_3$ are the characteristic exponents of $w_1$, $w_2$, and $w_3$ respectively, with respective triplets $(0, \sigma^2, 0)$, $(0,0, \One_{\abs{\cdot} > 1} \nu)$, and  $(0,0, \One_{\abs{\cdot} \leq 1} \nu)$.

(i) The characteristic exponent of $w_1$ is $\Psi_{1}(\xi) = - \sigma^2 \xi^2 /2$, hence $ \indlocinf (w_1) = \indlocsup (w_1) =  2$ follows directly from the definition of the indices in \eqref{eq:indass} and \eqref{eq:indasi}. Moreover, $\Psi_2$ is bounded due to $\lvert \Psi_2(\xi) \rvert \leq 2 \int_{\lvert t \rvert > 1} \mathrm{d} \nu(t) < \infty$, therefore $\indlocinf (w_2)  = \indlocsup (w_2)  = 0$. Finally, we have seen in the proof of Proposition \ref{prop:LIdecompo} that the moments of $w_3$ are finite, hence $\indmom(w_3) = \infty$. It is moreover clear that the moments of the Gaussian white noise are finite, hence $\indmom(w_1) = \infty$ and the relations \eqref{eq:firststuff} are proved.

{(ii)	Assume that $w_1 = 0$. 
	The characteristic exponent $\Psi_2$ is bounded by some constant $C >0$. Hence, since $\Psi(\xi) = \Psi_2(\xi) + \Psi_3(\xi)$, we deduce that
	\begin{equation} \label{eq:tructruc}
	    |\Psi_3(\xi) | - C \leq |\Psi(\xi) | \leq  |\Psi_3(\xi) | + C
	\end{equation}
	for every $\xi\in \R$. The left inequality~\eqref{eq:tructruc} implies that $\indlocinf  (w) \geq \indlocinf (w_3)$ and $\indlocsup (w) \geq \indlocsup  (w_3)$. The right inequality gives the other inequalities for the Blumenthal--Getoor indices and therefore $\indlocinf  (w) = \indlocinf (w_3)$ and $\indlocsup (w) = \indlocsup  (w_3)$.}

    	{For the moment index, we recall that $\abs{a+b}^p \leq 2^{p-1} (\abs{a}^p + \abs{b}^p)$ (by convexity of $x \mapsto x^p$ on $\R^+$) for every $a,b \in \R$ and $p \geq 1$ and that $\abs{a+b}^p \leq (\abs{a}^p + \abs{b}^p)$ when $0 <p <1$ (since $x\mapsto |x|^p$ is subadditive). We set $c_p = 2^{p-1}$ if $p\geq 1$ and $c_p = 1$ if $0<p<1$. }
    Then, if $X$ and $Y$ are two random variables such that $\mathbb{E}[\abs{Y}^p] < \infty$, then we have that
	\begin{equation} \label{eq:XplusY}
		c_p^{-1} \mathbb{E}[\abs{X}^p] - \mathbb{E}[\abs{Y}^p] 
		\leq \mathbb{E}[\abs{X + Y}^p]
		\leq
		c_p(\mathbb{E}[\abs{X}^p]  + \mathbb{E}[\abs{Y}^p]). 
	\end{equation}
	Applying \eqref{eq:XplusY} to $X  = \langle w_2 , \One_{[0,1]^d} \rangle$ and $Y = \langle  w_3 , \One_{[0,1]^d} \rangle$,  the later having finite $p$th moments for any $p>0$, we deduce that
	\begin{equation}
	\mathbb{E} \left[ \abs{\langle w ,  \One_{[0,1]^d} \rangle}^p\right] =\mathbb{E}[\abs{X+Y}^p]   < \infty \ \Longleftrightarrow \ \mathbb{E} \left[ \abs{\langle w_2 ,  \One_{[0,1]^d} \rangle}^p\right] = \mathbb{E}[\abs{X}^p]  < \infty.
	\end{equation}
	Hence, $w$ and $w_2$ have the same moment index. 


(iii) {Assume that $w_1 \neq 0$. Using that $\Psi_1(\xi) = -\sigma^2 \xi^2$} and that $\Psi_2$ and $\Psi_3$ (like every characteristic exponent, see for instance~\cite[Proposition 2.4]{Fageotthese}), is asymptotically dominated by $\xi \mapsto \xi^2$, we deduce that $C_1 \rvert \Psi_1 (\xi) \rvert \leq \lvert \Psi(\xi) \rvert \leq C_2 \rvert \Psi_1 (\xi) \rvert$ for some constants $C_1,C_2 > 0$ and every $\xi \in\R$ such that $\lvert \xi \rvert \geq 1$. Therefore, $w$ and $w_1$ have the same Blumenthal-Getoor indices $\indlocinf = \indlocsup = 2$. We have shown the equalities on Blumenthal-Getoor indices in \eqref{eq:secondstuff} and \eqref{eq:thirdstuff}. {The proof for the moment index is identical to the case $w_1 = 1$, this time with $Y = \langle w_1 + w_3 , \One_{[0,1]^d} \rangle$.}


\end{proof}

	\subsection{Weighted Besov Spaces} \label{subsec:besovspaces}
	
	As we have seen in Section~\ref{sec:related}, Besov spaces are natural candidates for characterizing the regularity of L\'evy processes and L\'evy white noises. 
	We define the family of weighted Besov spaces based on wavelet methods, as exposed in \cite{Triebel2008function}. 
	Besov spaces have a long history in functional analysis~\cite{Triebel2010theory}. They were successfully revisited by the introduction of wavelet methods following the works of Y. Meyer~\cite{MeyerWaO} and applied to the analysis of stochastic processes, including 
	the Brownian motion~\cite{ciesielski1993orlicz,Ciesielski1993quelques,Roynette1993},
	the fractional Brownian motion~\cite{Flandrin1992wavelet,Meyer1999wavelets},
	sparse random processes~\cite{Fageot2015wavelet,Fageot2017nterm,Pad2015optimality,Unser2014sparse},
	and general solutions of SPDEs~\cite{cioica2012spatial,cioica2010adaptive}.
	
	Essentially, weighted Besov spaces are subspaces of $\S'(\R^d)$ that are characterized by weighted sequence norms of the wavelet coefficients.
Following H. Triebel, we use the compactly supported wavelets discovered by I. Daubechies~\cite{Daubechies1988orthonormal}, which we introduce first.
The scale and shift parameters of the wavelets are respectively denoted by $j \in \N$ and $\bm{k}\in \Z^d$. 
The symbols $M$ and $F$ refer to the \emph{gender} of the wavelet ($M$ for the mother wavelets and  $F$ for the father wavelet). 
Consider two  functions $\psi_{M}$ and $\psi_{F} \in L_2(\R)$.
We set $\G^0 = \{ M, F \}^{d}$ and for $j\geq 1$, $\G^j = \{M,F\}^{d} \backslash \{ (F, \ldots , F) \} $.
Therefore,  the cardinal of $\G^0$ is $\mathrm{Card}(\G^0) = 2^d$, while $\mathrm{Card}(\G^j) = 2^d - 1$ for $j \geq 1$.
For $\bm{G} = (G_1,\ldots, G_d) \in \G^0$, called a gender, we set, for every $\bm{x} = (x_1, \ldots , x_d) \in \R^d$, $\psi_{\bm{G}}(\bm{x} ) = \prod_{i=1}^d \psi_{G_i}(x_i)$. 
For $j \in \N$, $\bm{G} \in\G^j$, and $\bm{k} \in  \Z^d$, we define 
\begin{equation}
	\psi_{j,\bm{G},\bm{k}} := 2^{jd/2} \psi_{\bm{G}}(2^j \cdot - \bm{k}).
\end{equation}
We shall also use the notations $\bm{F} = (F,\ldots ,F)$ and $\bm{M}= (M, \ldots , M)$ for the purely father and purely mother genders.
It is known that, for any $r_0 \geq 1$, there exists two functions $\psi_M, \psi_F \in L_2(\R)$, called \emph{Daubechies wavelets}, that are compactly supported, with at least $r_0$ continuous derivatives and vanishing moments up to order at least $(r_0-1)$, and such that the family\footnote{There is a slight abuse of notation when we write $(j,\bm{G},\bm{k})\in \N \times \G^j \times \Z^d$, since $j$ appears as the first element of the triplet $(j,\bm{G},\bm{k})$ and specifies the location of the gender $\bm{G} \in \G^j$. We keep this notation for its convenience.}
$	\{ \psi_{j,\bm{G},\bm{k}} \}_{(j,\bm{G},\bm{k})\in \N \times \G^j \times \Z^d}$
is an orthonormal basis of $L_2(\R^d)$ \cite[Section 1.2.1]{Triebel2008function}.

We now introduce the family of weighted Besov spaces $B_{p}^{\reg}(\R^d;\rate)$. Traditionally, Besov spaces also depend on the additional parameter $q \in (0,\infty]$ (see for instance \cite[Definition 1.22]{Triebel2008function}).  We shall only consider the case $q=p$ in this paper, so that we do not refer to this parameter.

We introduce weighted Besov spaces in Definition \ref{def:besovspaces} relying on the wavelet decomposition of (generalized) functions. This construction is equivalent to the more usual Fourier-based definitions, as proved in \cite[Theorem 1.26]{Triebel2008function}.  We use the notation $(x)_+ = \max(x,0)$. 

\begin{definition}\label{def:besovspaces}
	Let $\reg, \rate \in \R$ and $0< p \leq \infty$. Fix an integer $r_0 > \max ( \tau , d(1 / p - 1)_+ - \tau )$ and consider a family of Daubechies wavelets $\{ \psi_{j,\bm{G},\bm{k}} \}_{(j,\bm{G},\bm{k})\in \N \times \G^j \times \Z^d}$, where $\psi_M$ and $\psi_F$ have at least $r_0$ continuous derivatives and $\psi_M$ has vanishing moments up to order at least $(r_0-1)$.
	The \emph{weighted Besov space} $B_p^{\reg}(\R^d;\rate)$ is the collection of tempered generalized functions $ f \in \S'(\R^d)$ that can be written as
	\begin{equation} \label{eq:fconvuncond}
		f =   \sum_{j \in \N}   \sum_{\bm{G} \in\G^j}  \sum_{\bm{k}\in \Z^d}  c_{j,\bm{G},\bm{k}} \psi_{j,\bm{G},\bm{k}},
	\end{equation}
	where the $c_{j,\bm{G},\bm{k}}$ satisfy
	\begin{equation}
	 \sum_{j \in \N} 2^{j(\reg p -d + \frac{dp}{2})} 
	 	\sum_{\bm{G} \in\G^j}
		\sum_{\bm{k}\in \Z^d} \langle 2^{-j} \bm{k}  \rangle^{\rate p}  \lvert c_{j,\bm{G},\bm{k}} \rvert^p < \infty
		\end{equation}
and  	where the convergence \eqref{eq:fconvuncond} holds  on $\S'(\R^d)$. The usual adaptation is made for $p=\infty$; that is,
 	\begin{equation}
		\sup_{j \in \N} 2^{j ( \tau + \frac{d}{2} )} 
	 	\sup_{\bm{G} \in\G^j}
		\sup_{\bm{k}\in \Z^d} \langle 2^{-j} \bm{k}  \rangle^{\rate  }  \lvert c_{j,\bm{G},\bm{k}} \rvert  < \infty.
		\end{equation}
\end{definition}

The integer $r_0$ in Definition \ref{def:besovspaces} is chosen such that the mother wavelet has enough vanishing moments and the mother and father wavelets are regular enough to be applied to a function of $B_p^{\reg}(\R^d;\rate)$. We refer the reader to~\cite[Section 1.2.1]{Triebel2008function} and references therein for more details about the role of the smoothness and the vanishing moments of Daubechies wavelets.
When the convergence \eqref{eq:fconvuncond} occurs, the duality product $\langle f, \psi_{j,\bm{G},\bm{k}} \rangle$ is well defined and we have $c_{j,\bm{G},\bm{k}} = \langle f, \psi_{j,\bm{G},\bm{k}} \rangle$. Moreover, for $p < \infty$, the quantity
\begin{equation} \label{eq:besovnorm}
	\lVert f \rVert_{B_p^\reg(\R^d;\rate)} := \left(  \sum_{j \in \N} 2^{j(\reg p -d + \frac{dp}{2})} 
	 	\sum_{\bm{G} \in\G^j}
		\sum_{\bm{k}\in \Z^d} \langle 2^{-j} \bm{k}  \rangle^{\rate p}  \lvert \langle f , \psi_{j,\bm{G},\bm{k}} \rangle \rvert^p \right)^{1/p}	
\end{equation}
is finite for any $f \in B_p^\tau(\R^d;\rate)$ and specifies a norm (a quasi-norm, respectively) on the space $ B_p^\tau(\R^d;\rate)$, with $p\geq 1$ ($p<1$, respectively). The space $ B_p^\tau(\R^d;\rate)$ is a Banach (a quasi-Banach, respectively) for this norm (quasi-norm, respectively)~\cite[Theorem 1.26]{Triebel2008function}. For $p = \infty$, \eqref{eq:besovnorm} becomes
\begin{equation} \label{eq:pinftynorm}
\lVert f \rVert_{B_\infty^\reg(\R^d;\rate)} :=\sup_{j \in \N} 2^{j(\reg  + \frac{d }{2})} 
	 	\sup_{\bm{G} \in\G^j}
		\sup_{\bm{k}\in \Z^d} \langle 2^{-j} \bm{k}  \rangle^{\rate}  \lvert \langle f , \psi_{j,\bm{G},\bm{k}} \rangle \rvert,
\end{equation}
and $ B_\infty^\tau(\R^d;\rate)$ is a Banach space for this norm. 

\begin{proposition}[Embeddings between weighted Besov spaces] \label{prop:besovembed}
	Let $0 < p_0 \leq p_1 \leq \infty$ and 		
		$\reg_0 , \reg_1 , \rate_0 , \rate_1 \in \R$.
	\begin{itemize}
		\item We have the embedding $B_{p_0}^{\reg_0}(\R^d;\rate_0) \subseteq B_{p_1}^{\reg_1}(\R^d;\rate_1)$ as soon as
		\begin{equation} \label{eq:embedd1}
		{	\reg_0 - \reg_1 \geq \frac{d}{p_0} - \frac{d}{p_1}  \ \text{ and } \ \rate_0 \geq  \rate_1.}
		\end{equation}
		\item We have the embedding $B_{p_1}^{\reg_1}(\R^d;\rate_1) \subseteq B_{p_0}^{\reg_0}(\R^d;\rate_0)$ as soon as
		\begin{equation} \label{eq:embedd2}
			{\rate_1 - \rate_0 > \frac{d}{p_0} - \frac{d}{p_1}  \ \text{ and } \ \reg_1 > \reg_0.}
		\end{equation}
	\end{itemize}	
\end{proposition}

	{The embedding for the conditions \eqref{eq:embedd1} was proved by D.E. Edmunds and H. Triebel~\cite[Equation (9), Section 4.2.3]{Edmunds2008function} for general weights. The embedding for the conditions \eqref{eq:embedd2} was obtained in \cite[Section 2.2.2]{Fageot2017multidimensional}. Note that \eqref{eq:trivialembed} is deduced from  \eqref{eq:embedd1} by taking $p_0=p_1=p$.}
	The embedding relations are summarized in the two Triebel diagrams\footnote{The representation of the smoothness properties in diagrams with axis $(1/p,\tau)$ is inherited from the work of H. Triebel. It is very convenient because the smoothness has often a simple formulation in terms of $1/p$, as appear typically in our Theorem~\ref{maintheo:regdecaylimit}. This is also valid for the asymptotic rate $\rho$.} of Figure \ref{fig:embeddings}. 
	
\begin{figure}[t!] 
\centering
\begin{subfigure}[b]{0.30\textwidth}
\begin{tikzpicture}[x=2cm,y=2cm,scale=0.34]

\fill[green, opacity= 0.3] (0,-1) -- (2,1) -- (3,1) -- (3,-3/2) -- (0,-3/2) -- cycle; 
\draw[green, thick, dashed,->](2,1) --(3,1);
\draw[green, thick, dashed,- ](0,-1) -- (2,1);
\draw[green, thick,->](0,-1) -- (0,-3/2);

\fill[red, opacity= 0.3] (0,1) -- (2,1) -- (3,2) -- (0,2) -- cycle; 
\draw[red, thick, dashed,->](2,1) --(3,2);
\draw[red, thick, dashed,- ](2,1) -- (0,1);
\draw[red, thick,->](0,1) -- (0,2);

\draw[thick, ->] (0,0)--(3,0) node[circle,right] {$\frac{1}{p}$} ;
\draw[thick, ->] (0,-3/2)--(0,2) node[circle,above] {$\reg$} ;

\draw[ thick,color=black]
(-0.05,1) -- (0.05,1)  node[black,left] { $\reg_0$};

\draw[ thick,color=black]
(2,-0.05) -- (2,0.05)  node[black,below] { $\frac{1}{p_0}$};

\end{tikzpicture}
\caption{The $(1/p,\reg)$-diagram for fixed $\rate_0$.}
\end{subfigure}
\begin{subfigure}[b]{0.30\textwidth}
\begin{tikzpicture}[x=2cm,y=2cm,scale=0.34]

\fill[green, opacity= 0.3] (0,1) -- (1,1) -- (3,-1) -- (3,-3/2) -- (0,-3/2) -- cycle; 
\draw[green, thick, dashed,->](1,1) --(3,-1);
\draw[green, thick, dashed,- ](0,1) -- (1,1);
\draw[green, thick,->](0,1) -- (0,-3/2);

\fill[red, opacity= 0.3] (0,2) -- (1,1) -- (3,1) -- (3,2) -- cycle; 
\draw[red, thick, dashed,->](1,1) --(3,1);
\draw[red, thick, dashed,- ](0,2) -- (1,1);

\draw[thick, ->] (0,0)--(3,0) node[circle,right] {$\frac{1}{p}$} ;
\draw[thick, ->] (0,-3/2)--(0,2) node[circle,above] {$\rate$} ;

\draw[ thick,color=black]
(-0.05,1) -- (0.05,1)  node[black,left] { $\rate_0$};

\draw[ thick,color=black]
(1,-0.05) -- (1,0.05)  node[black,below] { $\frac{1}{p_0}$};

\end{tikzpicture}
\caption{The $(1/p,\rate)$-diagram for fixed $\tau_0$.}
\end{subfigure}

\caption{Representation of the embeddings between Besov spaces: If $f \in B_{p_0}^{\reg_0}(\R^d;\rate_0)$, then $f$ is in every Besov space that is in the lower shaded green regions. Conversely, if $f \notin B_{p_0}^{\reg_0}(\R^d;\rate_0)$, then $f$ is in none of the Besov spaces of the upper shaded red  regions.}
\label{fig:embeddings}
\end{figure}
	
	As a simple example, we obtain the Besov localization of the Dirac distribution. This result is of course well-known (an alternative proof can be found for instance in \cite{schmeisser87}) but we provide a new proof for two reasons: (1) it illustrates how to use the wavelet-based characterization of Besov spaces and (2) the result will be used to obtain sharp results for compound Poisson white noises.
	
	\begin{proposition} \label{prop:DiracBesov}
	Let $0< p < \infty$, $\tau \in \R$, and $\rho \in \R$.
 	Then, the Dirac impulse $\delta$ is in $B_p^\reg(\R^d;\rate)$ if and only if $\reg < \frac{d}{p} - d$. Moreover, $\delta \in B_{\infty}^{\tau}(\R^d; \rho)$ if and only if $\tau \leq - d$. 
	\end{proposition}
	
	We remark that the weight $\rho \in \R$ plays no role in Proposition \ref{prop:DiracBesov}. This is a simple consequence of the fact that $\delta$ is compactly supported, and therefore insensitive to the weight, as will appear in the proof.
	
	\begin{proof}[Proof of Proposition \ref{prop:DiracBesov}]
		We first treat the case $p < \infty$.
		The wavelet coefficients of $\delta$ are $c_{j,\bm{G},\bm{k}} = 2^{jd/2} \psi_{\bm{G}}(-\bm{k})$, hence 
		the Besov (quasi-)norm of the Dirac impulse is given by
		\begin{equation} \label{eq:normdirac}
			\lVert \delta \rVert^p_{B_p^\reg(\R^d;\rate)} 
			= 
			\sum_{j \in \N} 2^{j(\reg p - d + dp)} \sum_{\bm{G} \in\G^j} \sum_{\bm{k}\in \Z^d} \langle 2^{-j} \bm{k} \rangle^{\rate p} \lvert \psi_{\bm{G}}(- \bm{k}) \rvert^p.
		\end{equation}
		
		We first introduce some notations. 
		The $2^d$ wavelets $\psi_{\bm{G}}$ with gender $\bm{G}$ describing $\G^0$ are bounded, hence the constant $b = \max_{\bm{G} \in\G^0} \lVert \psi_{\bm{G}} \rVert_\infty$ is finite. We denote by $K$ the set of multi-integers $\bm{k} \in  \Z^d$ such that $\psi_{\bm{G}}(-\bm{k}) \neq 0$ for some gender $\bm{G} \in \mathcal{G}^0$. The set $K$ is finite because the $2^d$ wavelets are compactly supported and we set $n_0= \mathrm{Card} (K)$. Moreover, the set $K$ is non empty; otherwise, \eqref{eq:normdirac} would imply that   $\lVert \delta  \rVert^p_{B_p^\reg(\R^d;\rate)}  = 0$, hence $\delta = 0$, which is absurd. We fix some element $\bm{k}_0 \in K$ and $\bm{G}_0$ a gender such that $\psi_{\bm{G}_0}(-\bm{k}_0) \neq 0$. We set $a = \lvert \psi_{\bm{G}_0}(-\bm{k}_0) \rvert > 0$.
	Then, there exists a constant $M > 0$ such that $\lVert 2^{-j} \bm{k}\rVert \leq M$ for any $j \in \N$ and $(-\bm{k}) \in K$.
		In particular, for such $\bm{k}$ and $j$, we have that $1 \leq \langle 2^{-j} \bm{k} \rangle = (1 + \lVert 2^{-j} \bm{k} \rVert^2)^{1/2} \leq \langle M \rangle$, and therefore,
		\begin{equation}
			0 < \min(1,\langle M \rangle^{\rho p}) \leq    \langle 2^{-j} \bm{k} \rangle^{\rate p}  \leq  \max(1,\langle M \rangle^{\rho p})  < \infty.
		\end{equation}	
	Fix $j\geq 0$. 
	Then,  we have the lower bound
	\begin{equation}\label{eq:Afordirac}
	\sum_{\bm{G} \in\G^j} \sum_{\bm{k}\in \Z^d} \langle 2^{-j} \bm{k} \rangle^{\rate p} \lvert \psi_{\bm{G}}(- \bm{k}) \rvert^p \geq \langle 2^{-j} \bm{k}_0\rangle^{\rho p} a^p \geq  A:=  \min(1,\langle M \rangle^{\rho p}) a^p.
	\end{equation} 
	We recall that $\mathrm{Card}(\G^j) \leq 2^d$. Then, we also have the upper bound
	\begin{equation}\label{eq:Bfordirac}
	\sum_{\bm{G} \in \G^j} \sum_{\bm{k}\in \Z^d} \langle 2^{-j} \bm{k} \rangle^{\rate p} \lvert \psi_{\bm{G}}(- \bm{k}) \rvert^p \leq
	\sum_{\bm{G} \in \G^j} \sum_{\bm{k}\in K} \langle 2^{-j} \bm{k} \rangle^{\rate p} b^p
	\leq B:= 2^d n_0 \max(1,\langle M \rangle^{\rho p}) b^p.
	\end{equation} 	
	Combining \eqref{eq:Afordirac} and \eqref{eq:Bfordirac}, we therefore deduce that
	\begin{equation}
			A \sum_{j \in \N} 2^{j (\reg p - d + dp)} \leq \lVert \delta \rVert^p_{B_p^\reg(\R^d;\rate)} \leq B  \sum_{j \in \N} 2^{j (\reg p - d + dp)}.
		\end{equation}
		The sum converges for $(\reg p  - d + dp)  < 0$ and diverges otherwise, implying the result. \\
		
		We now adapt the argument to the case $p = \infty$ for which the Besov norm is given by  
		\begin{equation}
		\lVert \delta \rVert_{B_\infty^\reg(\R^d;\rate)}  = \sup_{j \in \N} 2^{j (\tau+ d)} \sup_{\bm{G} \in \mathcal{G}^j} \sup_{\bm{k}\in \Z^d} \langle 2^{-j} \bm{k} \rangle^\rho \lvert \psi_{\bm{G}} (-\bm{k} ) \rvert. 
		\end{equation}
		We have that, for any $j \in \N$, 
		\begin{equation}
			A' := a \min (1, \langle M \rangle^\rho) 
			\leq
			\sup_{\bm{G} \in \mathcal{G}^0} \sup_{\bm{k}\in \Z^d} \langle 2^{-j} \bm{k} \rangle^\rho \lvert \psi_{\bm{G}}(-\bm{k}) \rvert
			\leq 
			B' := b \max (1, \langle M \rangle^\rho). 
		\end{equation}
		Therefore, we deduce that
			\begin{equation}
			A' \sup_{j \in \N} 2^{j (\reg  + d )} \leq \lVert \delta \rVert_{B_\infty^\reg(\R^d;\rate)} \leq B'  \sup_{j \in \N} 2^{j (\reg  + d)},
		\end{equation}
		and $\lVert \delta \rVert_{B_\infty^\reg(\R^d;\rate)} < \infty$ if and only if $\tau + d \leq 0$. 		
	\end{proof}
	
\section{Gaussian White Noise} \label{sec:Gaussian}

Our goal in this section is to prove the Gaussian part of Theorem \ref{maintheo:regdecaylimit}.
Without loss of generality, we focus on the Gaussian white noise with zero mean and unit variance. 

The Gaussian case is much simpler than the general one since the wavelet coefficients of the Gaussian white noise are independent and identically distributed.
We present it separately for three reasons: (i) it can be considered as an instructive toy problem that already contains  some of the technicalities that will appear for the general case; (ii) it cannot be deduced from the other sections, where the results are based on a careful study of the L\'evy measure; 
and (iii) the localization of the Gaussian white noise in \emph{weighted} Besov spaces has not been addressed in the literature, to the best of our knowledge.
We first state three simple lemmata that will be useful throughout the paper.

{\begin{lemma} \label{lemma:sumiidnew}
	Let $(N_k)_{k\geq 1}$ be a sequence such that $N_k \rightarrow \infty$ when $k \rightarrow \infty$.
	Assume that we have i.i.d. random variables $Y_k^i$ with $k \geq 1$ and $1 \leq i \leq N_k$ such that $Y_k^i \geq 0$ and $\mathscr{P}( Y_k^i  > 0) > 0$. 	
	Let $(Z_k)_{k\geq 1}$ be the family of random variables such that
	$	Z_k = \frac{1}{N_k} \sum_{i=1}^{N_k} Y_{k}^{i}$.
	Then, $\sum_{k\geq 1} Z_k =   \infty$ a.s. \\
\end{lemma}}

\begin{proof}
 Let $\mu = \mathbb{E}[Y_{k}^{i}] \in (0, \infty]$. If $\mu = \infty$, we set $\tilde{Y}_{k}^{i} =  \min (Y^i_k, 1)$. Then, the $\tilde{Y}_{k}^{i}$ are i.i.d., non-negative with $\mathscr{P}( \tilde{Y}_k^i  > 0) > 0$, and such that $\tilde{\mu} = \mathbb{E}[ \tilde{Y}_{k}^{i} ] \leq 1 < \infty$. 
 Moreover, we have that
 \begin{equation} \label{eq:petiteborne}
 \sum_{k\geq 1} Z_k
 =   \sum_{k\geq 1}\frac{1}{N_k} \sum_{i=1}^{N_k} Y_{k}^{i} \geq 
  \sum_{k\geq 1}\frac{1}{N_k} \sum_{i=1}^{N_k} {\tilde{Y}}_{k}^{i}
  =  \sum_{k\geq 1} \tilde{Z}_k,
 \end{equation}
 where we set $\tilde{Z}_k = \frac{1}{N_k} \sum_{i=1}^{N_k} {\tilde{Y}}_{k}^{i}$.
Due to \eqref{eq:petiteborne}, it is then sufficient to demonstrate Lemma \ref{lemma:sumiidnew} for $\mu < \infty$, which we do now.
 
 Let $(W_k)_{k\geq 1}$ be a family of i.i.d. random variables whose common law is the one of the $Y_{k}^{i}$ and define $\overline{W}_k = \frac{1}{N_k} \sum_{i=1}^{N_k} W_i$. 
 The weak law of large numbers implies that $\mathscr{P} ( | \overline{W}_k - \mu | \geq x)$ vanishes when $k \rightarrow \infty$ for any $x>0$. Taking $x = \mu/2$, we readily deduce that $\mathscr{P}( \overline{W}_k \geq \mu/2)$ goes to $1$ when $k\rightarrow \infty$.
 Moreover, we have the equality $Z_k \overset{(\mathcal{L})}{=} \overline{W}_k$, therefore, we also have that 
 \begin{equation}
 \mathscr{P}( Z_k \geq \mu/2) \underset{k\rightarrow \infty}{\longrightarrow} 1.
 \end{equation}
 This implies in particular that $\sum_{k\geq1}  \mathscr{P}( Z_k \geq \mu/2) = \infty$. The events $\{Z_k \geq \mu/2\}$ are moreover independent  due to the independence of the $Y^i_k$. Using the Borel-Cantelli lemma, we deduce that $Z_k \geq \mu/2$ for infinitely many $k$ a.s. An obvious consequence is then that $\sum_{k\geq 1} Z_k =  \infty$ a.s..
\end{proof}
		
As a consequence of Lemma \ref{lemma:sumiidnew}, we deduce Lemma \ref{lemma:sumiid}.
		
\begin{lemma} \label{lemma:sumiid}		
	Assume   that $(X_{\bm{k}})_{\bm{k} \in \Z^d}$, is a sequence of i.i.d. random variables such that $\mathscr{P}( |X_{\bm{k}}| > 0) > 0$.
	 Then,
	\begin{equation} \label{eq:explosemoica}
		\sum_{\bm{k} \in \Z^d} \frac{\lvert X_{\bm{k}} \rvert}{\langle \bm{k}\rangle^d} =  \infty \text{ a.s.}
	\end{equation}
\end{lemma}

\begin{proof}
	First of all, the result for any dimension $d$ is easily deduced from the one-dimensional case. Moreover, $\lvert k \rvert$ and $\langle k \rangle$ are equivalent asymptotically, hence it is equivalent to show that $\sum_{k \geq 1} \frac{\lvert X_k \rvert}{k} = \infty$ for $X_k$ i.i.d.  with $\mathscr{P}( |X_k| > 0) > 0$. {Setting $Z_k = \frac{1}{2^{k-1}} \sum_{\ell =2^{k-1}}^{2^{k}-1} \lvert X_\ell \rvert$ for $k \geq 1$, we deduce \eqref{eq:explosemoica} by applying Lemma~\ref{lemma:sumiidnew} with $N_k = 2^{k-1}$ and $Y^i_k = X_{2^{k-1} + i-1}$ and observing that $\sum_{k \geq 1} \frac{\lvert X_k \rvert}{k} \geq \frac{1}{2} \sum_{k\geq 1} Z_k = \infty$.}
\end{proof}

Finally, we state the last lemma that deals with supremum of i.i.d. sequences of random variables.

\begin{lemma}\label{lemma:third}
Let $(X_k)_{k \geq 1}$ be a sequence of i.i.d. random variables such that $\mathscr{P}(|X_k| \geq M) > 0$ for every $M \geq 0$. Then, we have that, almost surely,
\begin{equation} \label{eq:trucsimple}
	\sup_{k\geq 1} |X_k| = \infty.
\end{equation}
\end{lemma}

\begin{proof}
Let $M > 0$. The assumption $\mathscr{P}(|X_k| \geq M) > 0$ and the fact that the events $\{|X_k| \geq M\}$ are independent implies, thanks to the Borel-Cantelli lemma, that there exists almost surely (infinitely many) $k \geq 1$ such that $|X_k| \geq M$. Hence, $\sup_{k\geq 1} |X_k| \geq M$ almost surely. This being true for every $M >0$, we deduce \eqref{eq:trucsimple}.
\end{proof}

We characterize the Besov regularity of the Gaussian white noise in Proposition \ref{prop:GaussBesov}. 

\begin{proposition} \label{prop:GaussBesov}
	Fix $0< p \leq \infty$ and $\tau, \rho \in \R$.	
	The Gaussian white noise $w$ is 
	\begin{itemize}
	\item almost surely in $B_p^\reg(\R^d;\rate)$ if $\reg < - d/2$ and $\rate < - d /p$, and 
	\item almost surely not in $B_p^\reg(\R^d;\rate)$ if $\reg \geq  - d/2$ or $\rate \geq - d /p$.
	\end{itemize}
\end{proposition}

A direct consequence of Proposition \ref{prop:GaussBesov} is Corollary \ref{coro:GaussBesov}, where we identify the  local smoothness and the  asymptotic growth rate of the Gaussian white noise, that are defined for (deterministic) generalized functions in Section \ref{sec:taurho}.

\begin{corollary} \label{coro:GaussBesov}
	Let $w$ be a Gaussian white noise and $0 < p \leq \infty$. Then, we have almost surely that
	\begin{equation} \label{eq:taurhogauss}
	\tau_p(w) = -\frac{d}{2} \quad \text{ and } \quad \rho_p(w) = - \frac{d}{p}.
	\end{equation}
\end{corollary}

\textit{Remark.} The determination of the local smoothness and the asymptotic growth rate is insensitive to the fact that the generalized function $f$ is or is not in the critical space $B_p^{\tau_p(f)}(\R^d;\rho_p(f))$. For the Gaussian white noise, Proposition \ref{prop:GaussBesov} implies that (almost surely)  $w \notin B_p^{\tau_p(w)}(\R^d;\rho_p(w))$ for every $0 < p \leq \infty$. In that sense, Proposition \ref{prop:GaussBesov} contains more information than Corollary \ref{coro:GaussBesov}, since we cannot deduce the critical cases treated in the proposition from the result of the corollary.

\begin{proof}[Proof of Proposition \ref{prop:GaussBesov}]
Recall that we restrict, without loss of generality, to Gaussian white noise with unit variance $\sigma^2 = 1$. Then, $\langle w, \varphi_1 \rangle$ and $\langle w, \varphi_2\rangle$ are independent if and only if $\langle \varphi_1 ,\varphi_2 \rangle = 0$. Moreover, $\langle w, \varphi \rangle$ is a Gaussian random variable with variance $\lVert \varphi \rVert_2^2$ \cite[Section 2.5]{GelVil4}. The family of functions $	\{ \psi_{j,\bm{G},\bm{k}} \}_{(j,\bm{G},\bm{k})\in \N \times \G^j \times \Z^d}$ being orthonormal, the random variables $\langle w , \psi_{j,\bm{G},\bm{k}} \rangle$  are therefore i.i.d. with law $\mathcal{N}(0,1)$.  \\

	\textbf{Case $p < \infty$, $\reg < - d/2$, and $\rate < - d /p$.}
	For $p > 0$, we denote by $C_p$ the $p$th moment of a Gaussian random variable with zero mean and unit variance. In particular, we have that  $\mathbb{E} \left[\lvert \langle w , \psi_{j,\bm{G},\bm{k}} \rangle \rvert^p \right] = C_p$  for any $j \in \N, \bm{G} \in \mathcal{G}^j, \bm{k} \in \Z^d$, and therefore  
	\begin{align}
	\mathbb{E} \left[\|w  \|_{B_p^\reg(\mathbb{R}^d;\rate)}^p \right] 
	&= 
	 \sum_{j \in \N} 2^{j(\reg p -d + \frac{dp}{2})} 
	 	\sum_{\bm{G} \in \G^j}
		\sum_{\bm{k}\in \Z^d} \langle 2^{-j} \bm{k}  \rangle^{\rate p}  \mathbb{E} \left[\lvert \langle w , \psi_{j,\bm{G},\bm{k}} \rangle \rvert^p \right]	 
 	\nonumber   \\ 	
	& = 
	C_p \sum_{j \in \N} 2^{j(\reg p -d + \frac{dp}{2})}  \mathrm{Card}(\G^j) \sum_{\bm{k}\in \Z^d} \langle 2^{-j} \bm{k}  \rangle^{\rate p} \nonumber \\
	& \leq 2^d C_p \sum_{j \in \N} 2^{j(\reg p -d + \frac{dp}{2})}  \sum_{\bm{k}\in \Z^d} \langle 2^{-j} \bm{k}  \rangle^{\rate p}.
	\end{align}
	The last inequality is due to $\mathrm{Card} ( \G^j ) \leq 2^d$. 
	Since $\rate p < -d$ and 
	$\langle 2^{-j} \bm{k} \rangle \underset{\lVert \bm{k}\rVert \rightarrow \infty}{\sim} 2^{-j} \lVert \bm{k} \rVert $, 
	we have that 
	$\sum_{\bm{k}\in \Z^d} \langle 2^{-j} \bm{k} \rangle^{\rate p} < \infty$. 
	Moreover, we recognize a Riemann sum and have the convergence
	\begin{equation}
		2^{-jd} \sum_{\bm{k}\in \Z^d} \langle 2^{-j} \bm{k} \rangle^{\rate p}  \underset{j \rightarrow \infty}{\longrightarrow} \int_{\R^d} \langle \bm{x} \rangle^{\rate p} \drm \bm{x} < \infty.
	\end{equation}
	In particular, the series $\sum_j 2^{j(\reg p  + \frac{dp}{2})}  \left( 2^{-jd} \sum_{\bm{k}} \langle 2^{-j} \bm{k}  \rangle^{\rate p}\right)$ converges if and only if the series 
	$\sum_j 2^{j(\reg p   + \frac{dp}{2})}$ does; in other words, if and only if $\reg <  - d / 2$. Finally, if $\reg <  - d / 2$ and $\rate < - d / p$, we have  shown that $\mathbb{E}[\|w  \|_{B_p^\reg(\mathbb{R}^d;\rate)}^p]  < \infty$ and therefore $w \in {B_p^\reg(\mathbb{R}^d;\rate)}$ almost surely. \\
	
	\textbf{Case $p < \infty$ and $\reg \geq - d/2 $.}
	Then, we have $2^{j (\reg p - d + dp/2)} \geq 2^{-jd}$. 
	We aim at establishing a lower bound for the  Besov norm of $w$ and we restrict to the purely mother wavelet with gender $\bm{G} = \bm{M} = (M,\ldots, M) \in \G^j$ for any $j \in \N$. 
	For $\bm{k} = (k_1, \ldots, k_d) \in \Z^d$ such that $0 \leq k_i < 2^j$ for every $i = 1, \ldots, d$, we have that 
	\begin{equation}
	\langle 2^{-j} \bm{k} \rangle^{\rho p} \geq C := \min_{\lVert \bm{x} \rVert_\infty \leq 1} \langle \bm{x} \rangle^{\rho p}.
	\end{equation}
	For $\lVert \bm{x} \rVert_\infty \leq 1$, we have that 
	$1 \leq \langle \bm{x} \rangle = (1 + \lVert \bm{x} \rVert^2)^{1/2} \leq (1+d)^{1/2}$, hence 
	$C  \geq \min(1, (1+d)^{\rho p/2}) > 0$. 
	Then, we have
	\begin{align} \label{eq:perlinpinpin}
		\lVert w \rVert^p_{B_p^{\reg}(\R^d;\rate)} &\geq   \sum_{j \in \N} 2^{-jd} \sum_{0\leq k_1, \ldots , k_d < 2^{j}} \langle 2^{-j} \bm{k} \rangle^{\rho p} \lvert \langle w , \psi_{j, \bm{M} ,\bm{k}} \rangle \rvert^p \nonumber \\
		&\geq C  \sum_{j \in \N} 2^{-jd} \sum_{0\leq k_1, \ldots , k_d < 2^{j}}   \lvert \langle w , \psi_{j, \bm{M} ,\bm{k}} \rangle \rvert^p.
	\end{align}
	The random variables $\langle w , \psi_{j, \bm{M} ,\bm{k}} \rangle$ are i.i.d. We can therefore apply Lemma \ref{lemma:sumiidnew} with blocks of size $2^{jd}$, which goes to infinity when $j \rightarrow \infty$ to conclude that $\lVert w \rVert^p_{B_p^{\reg}(\R^d;\rate)} = \infty$ a.s. \\
	
	\textbf{Case $p < \infty$ and $\rate \geq - d/p $.}
	We retain only the father wavelet $\phi = \psi_{\bm{F}}$ where $\bm{F} = (F,\ldots ,F) \in \G^0$ and the scale $j = 0$ and exploit the relation $\rate p \geq -d$ to deduce the lower bound 
	\begin{equation} \label{eq:trucgaussianpeugaussien}
	\lVert w \rVert^p_{B_p^{\reg}(\R^d;\rate)} 
	\geq 
	\sum_{\bm{k}\in \Z^d} \langle \bm{k} \rangle^{\rate p} \lvert \langle w, \phi_{\bm{k}} \rangle \rvert^p 
	\geq
	 \sum_{\bm{k} \in \Z^d} \frac{\lvert\langle w, \phi_{\bm{k}} \rangle \rvert^p}{\langle \bm{k} \rangle^d },
	\end{equation}
	with the notation $\phi_{\bm{k}} = \phi(\cdot - \bm{k})$. 
	Finally, the random variables $\langle w, \phi_{\bm{k}} \rangle$ being i.i.d., Lemma \ref{lemma:sumiid} applies and $\lVert w \rVert^p_{B_p^{\reg}(\R^d;\rate)}  = \infty$ almost surely. \\
	
	\textbf{Case $p = \infty$, $\tau < - d/2$, and $\rho< 0$.} This case is deduced using the embeddings between Besov spaces. 
	First of all, for $\epsilon \leq \min ( - d/2 - \tau , -\rho)$, we have that 
	\begin{equation} \label{eq:trutruc}
	B_{\infty}^{-d/2- \epsilon}(\R^d;-\epsilon)  \subseteq  B_{\infty}^{-d/2- \epsilon}(\R^d;\rho) \subseteq   B_{\infty}^{\tau}(\R^d;\rho) 
	\end{equation} 
	using \eqref{eq:trivialembed} first for the weight and then for the smoothness parameters. It therefore suffices to show the existence of $0 < \epsilon \leq \min ( - d/2 - \tau , -\rho)$ such that $ w \in B_{\infty}^{-d/2- \epsilon}(\R^d;-\epsilon)$.
	Fix such an $\epsilon$. Then, for every $p < \infty$, we already proved that $w \in B_p^{-d/2-\epsilon / 2}(\R^d; - d / p - \epsilon / 2)$ a.s. Applying this to $p = 2 d / \epsilon$, we then remark that 
	\begin{equation}
	B_p^{-d/2-\epsilon / 2}(\R^d; - d / \rho - \epsilon / 2) = B_{2d / \epsilon}^{-d/2 -\epsilon /2} (\R^d; - \epsilon).
	\end{equation}
	Moreover, applying Proposition \ref{prop:besovembed} with $p_0 = 2 d / \epsilon < p_1 = \infty$, $\tau_0 = - d / 2 - \epsilon / 2$, $\tau_1 = - d/2 - \epsilon$, $\rho_0 = -d / p_0 - \epsilon /2 = - \epsilon = \rho_1$, we easily verify that \eqref{eq:embedd1} is satisfied and therefore, using also \eqref{eq:trutruc}, we have a.s. that  
	\begin{equation}
	w \in B_{2d / \epsilon}^{-d/2 -\epsilon /2} (\R^d; - \epsilon) \subset B_{\infty}^{-d/2- \epsilon}(\R^d;-\epsilon) \subseteq   B_{\infty}^{\tau}(\R^d;\rho) ,
	\end{equation}
	concluding this case. \\
	
	\textbf{Case $p = \infty$ and $\tau \geq -d/2$.} Note that the case  $\tau > - d/2$ can be deduced from the results for $p < \infty$ by embedding, but one cannot deduce the case $\tau= -d/2$.  By keeping only the purely mother wavelet $\psi_{\bm{M}}$ with $\bm{M} = (M,\ldots, M)$ and the shift parameter $\bm{k} = \bm{0}$, the Besov norm \eqref{eq:pinftynorm} applied to the Gaussian white noise is
	\begin{equation} \label{eq:encoreunela}
	\lVert w \rVert_{B_\infty^{\tau}(\R^d;\rho)} \geq  \sup_{j \in \N}   \lvert \langle w , \psi_{j,\bm{M},\bm{0}} \rangle \rvert.
	\end{equation}
	The Gaussian random variables $\langle w , \psi_{j,\bm{M},\bm{0}} \rangle$ are independent and verifies the conditions of Lemma \ref{lemma:third} (since $\mathscr{P} (\lvert \langle w , \psi_{j,\bm{M},\bm{0}} \rangle \rvert \geq M) > 0$ for every $M \geq 0$). This implies that $ \sup_{j \in \N}   \lvert \langle w , \psi_{j,\bm{M},\bm{0}} \rangle \rvert = \infty$ a.s., and therefore $w \notin B_\infty^{\tau}(\R^d;\rho)$ a.s. due to \eqref{eq:encoreunela}. \\
	
	\textbf{Case $p = \infty$ and $\rho \geq 0$.} Again, the case  $\rho > 0$ can be deduced from the results for $p < \infty$ by embedding, but the case $\rho = 0$ cannot. 
	Using that $\langle \bm{k} \rangle^\rho \geq 1$ for $\rho \geq 0$ and keeping only the father wavelet $\phi = \psi_{\bm{F}}$ with $\bm{F} = (F,\ldots , F) \in \G^0$ and the scale $j = 0$, the Besov norm \eqref{eq:pinftynorm} of $w$ is lower bounded by
	\begin{equation} \label{eq:ploufplouf}
	\lVert w \rVert_{B_\infty^{\tau}(\R^d;\rho)} \geq \sup_{\bm{k}\in \Z} \lvert \langle w , \phi_{\bm{k}} \rangle \rvert.
	\end{equation}
	Again, Lemma \ref{lemma:third} applies to the Gaussian random variables $\langle w ,\phi_{\bm{k}} \rangle$ and  $\max_{\bm{k}\in \Z} \lvert \langle w ,\phi_{\bm{k}} \rangle \rvert = \infty$ a.s., implying that $\lVert w \rVert_{B_\infty^{\tau}(\R^d;\rho) }= \infty$ a.s. due to \eqref{eq:ploufplouf}. 
\end{proof}

The proof of Proposition \ref{prop:GaussBesov} for the case $\rho \geq - d / p$ uses an argument that can be easily adapted to any L\'evy white noise. We hence state this result in full generality.

\begin{proposition}\label{prop:uneptiteprop}
Fix $0<p \leq \infty$ and $\tau,\rho\in\R$.
If $w$ is a non-constant L\'evy white noise, then $w \notin B_{p}^\tau (\R^d;\rho)$ as soon as $\rho \geq -d /p$. 
\end{proposition}

\begin{proof}
The proof is very similar to the one of Proposition \ref{prop:GaussBesov} for $\rho \geq - d /p$ and $p < \infty$, and for $\rho \geq 0$ and $p = \infty$, except that we only consider father wavelet  $\phi = \psi_{\bm{F}}$ and its shifts $\phi_{\bm{k}} =   \psi_{0,\bm{F},\bm{k}}$ with $\bm{k} = k_0 \Z^d$, where $k_0 \in \N\backslash \{0\}$ is chosen such that the $\phi_{\bm{k}}$ have disjoint supports. In particular, this implies the random variables $\langle w , \phi_{\bm{k}} \rangle$, are independent  for $\bm{k} \in k_0 \Z^d$ (the support of the test functions being disjoint) and independent (the L\'evy white noise being stationary). As a consequence, \eqref{eq:trucgaussianpeugaussien} becomes 
	\begin{equation} \label{eq:trucgaussianpeugaussienbis}
	\lVert w \rVert^p_{B_p^{\reg}(\R^d;\rate)} 
	\geq 
	\sum_{\bm{k}\in k_0\Z^d} \langle \bm{k} \rangle^{\rate p} \lvert \langle w, \phi_{\bm{k}} \rangle \rvert^p 
	\geq
	 \sum_{\bm{k} \in k_0\Z^d} \frac{\lvert\langle w, \phi_{\bm{k}} \rangle \rvert^p}{\langle \bm{k} \rangle^d }
	\end{equation}
	and Lemma \ref{lemma:sumiid}  applies again, implying that $\lVert w \rVert^p_{B_p^{\reg}(\R^d;\rate)} =\infty$ a.s. Similarly, \eqref{eq:ploufplouf} becomes 
	\begin{equation} \label{eq:ploufplouf}
	\lVert w \rVert_{B_\infty^{\tau}(\R^d;\rho)} \geq \sup_{\bm{k}\in k_0 \Z} \lvert \langle w , \phi_{\bm{k}} \rangle \rvert.
	\end{equation}
	Then, it suffices to observe that $\mathscr{P}(  \lvert \langle w , \varphi \rangle \rvert \geq M ) > 0$ for any L\'evy white noise $w$ and any test function $\varphi \neq 0$. Indeed, the probability measure of an infinitely divisible is not compactly supported, except for constant L\'evy white noise $w = \mu \in\R$, what we have excluded~\cite[Corollary 24.4]{Sato1994levy}. Applying Lemma \ref{lemma:third}, we deduce finally that $w \notin B_{\infty}^\tau(\R^d;\rho)$ a.s.
\end{proof}

\section{Compound Poisson White Noise} \label{sec:Poisson}
			
	Compound Poisson white noises are almost surely made of countably many Dirac impulses, what will be crucial in their analysis.	
	Our positive results  are based on a careful  estimation of the moments of the compound Poisson white noise presented in  Proposition \ref{lemma:momentPoisson}.

	\begin{proposition} \label{lemma:momentPoisson}
		Let $w$ be a compound Poisson white noise with moment index $\indmom \in (0,\infty]$ and $0 < p < \indmom$.
		Then, there exists a constant $C$ such that
		\begin{equation} \label{eq:thecontrolPoisson}
			\mathbb{E} [ \abs{\langle w, \psi_{j,\bm{G},\bm{k}} \rangle}^p ] \leq C 2^{jpd/2 - jd}
		\end{equation}
		for every $j \in \N$, $\bm{G}\in \G^j$, and $\bm{k}\in \Z^d$.
	\end{proposition}

	\begin{proof}
	We recall that the Lebesgue measure is denoted by $\mathrm{Leb}$.
	Let $\lambda >0$ and $P$ be  respectively the sparsity parameter  and the law of the jumps of $w$. Then, we have that 
\begin{equation} \label{eq:poissonwithdirac}
w \overset{(\mathcal{L})}{=} \sum_{k\in \N} a_k \delta (\cdot - \bm{x}_k ),
\end{equation}
	where the $a_k$ are i.i.d. with law $P$, and the $\bm{x}_k$, independent from the $a_k$, are randomly located such that $\mathrm{Card} \{ k \in \N, \ \bm{x}_k \in B \}$ is a Poisson random variable with parameter $\lambda \mathrm{Leb}(B)$ for any Borel set $B\subset \R^d$ with finite Lebesgue measure. For a demonstration that the right term in \eqref{eq:poissonwithdirac} specifies a compound Poisson white noise in the sense of a generalized random process with the adequate characteristic functional, we refer the reader to  \cite[Theorem 1]{Unser2011stochastic}.
	
	 Let $\psi \in L_2(\R^d) \backslash\{0 \}$ be a compactly supported function and $I_{\psi}$ be the closed convex hull of its support. In particular, $0 < \mathrm{Leb}(I_{\psi}) < \infty$. We set $N(\psi) = \mathrm{Card} \{ k \in \N, \ \bm{x}_k \in I_{\psi}\}$, which is a Poisson random variable with parameter $\lambda \mathrm{Leb}(I_{\psi})$. We denote by $a'_n$ and $\bm{x}'_n$, $n=1,\ldots, N(\psi)$, the weights and Dirac locations of the compound Poisson white noise $w$ on $I_{\psi}$. 
	That is, $\langle w, \psi \rangle = \sum_{n=1}^{N(\psi)} a_n' \psi(\bm{x}_n')$.
	By conditioning on $N(\psi)$, we then have that
	\begin{align} \label{eq:equation1pourPoisson}
		\mathbb{E} [ \abs{\langle w, \psi \rangle}^p ]  
		& = 
		\sum_{N=1}^{\infty} \mathscr{P} ( N(\psi) = N ) \mathbb{E} \left[ \abs{\langle w, \psi \rangle}^p  | N(\psi) = N \right] 	 \nonumber \\
		&=
				\sum_{N=1}^{\infty} \mathscr{P} ( N(\psi) = N ) \mathbb{E} \left[ \abs{\sum_{n=1}^{N} a_n' \psi(\bm{x}_n')}^p  | N(\psi) = N \right] 	 \nonumber \\
		& \overset{(i)}{\leq}
		\sum_{N=1}^{\infty} \mathscr{P} ( N(\psi) = N ) \mathbb{E} \left[ N^{\max(0,p-1)} \sum_{n=1}^{N} \abs{ a_n' \psi(\bm{x}_n')}^p  | N(\psi) = N \right] 	 \nonumber \\
		& \leq \lVert \psi \rVert_{\infty}^p 
		\sum_{N = 1}^{\infty} 
		\left( N^{\max(0,p-1)} \mathscr{P} ( N(\psi) = N )
			\left( \sum_{n=1}^N \mathbb{E} \left[\abs{a_n'}^p  | N(\psi) = N \right] \right) \right) \nonumber \\
	&  \overset{(ii)}{=} \lVert \psi \rVert_{\infty}^p \sum_{N = 1}^{\infty}  \left( N^{\max(0,p-1)}  \mathscr{P} ( N(\psi) = N )
		\left( \sum_{n=1}^N \mathbb{E} \left[\abs{a_n'}^p\right]  \right) \right)  \nonumber \\
		&\overset{(iii)}{=} \lVert \psi \rVert_{\infty}^p\mathbb{E} \left[\abs{a_1'}^p\right] \sum_{N=1}^\infty N^{\max(1,p)} \mathscr{P} ( N(\psi) = N ) ,
	\end{align}
	where $(i)$ uses the relation $ \abs{\sum_{n=1}^N y_n}^p \leq N^{\max( 0, p-1)} \sum_{n=1}^N \abs{y_n}^p$, valid for any $p>0$ and $y_n \in \R$~\cite[Eq.(50)]{Fageot2016gaussian},  $(ii)$ is due to $\mathbb{E} \left[\abs{a_n'}^p  | N(\psi) = N \right] = \mathbb{E} \left[\abs{a_n'}^p \right]$, $a_n'$ and $N(\psi)$ being independent, and $(iii)$ comes from $\sum_{n=1}^N \mathbb{E} \left[\abs{a_n'}^p\right] = N \mathbb{E} \left[\abs{a_1'}^p\right]$, the $a_n'$ sharing the same law.
	
	Our goal is now to apply \eqref{eq:equation1pourPoisson} to $\psi = \psi_{j,\bm{G},\bm{k}}$.
	For fixed $j\geq 1$ and $\bm{k}\in \Z^d$, the Lebesgue measure of the convex hull $I_{\psi_{j,\bm{G},\bm{k}}}$ of the support of the $\psi_{j,\bm{G},\bm{k}}$ is $\mathrm{Leb}(I_{\psi_{j,\bm{G},\bm{k}}} ) = 2^{-jd}  \mathrm{Leb}(I_{\psi_{\bm{G}}})$.
	Therefore, $N(\psi_{j,\bm{G},\bm{k}}) = \mathrm{Card} \{ k \in \N, \ \bm{x}_k \in \psi_{j,\bm{G},\bm{k}}\}$ is a Poisson random variable with   parameter $2^{-jd} \lambda \mathrm{Leb}(I_{\psi_{\bm{G}}})$. As a consequence, we have
	\begin{align} \label{eq:equation2pourPoisson}
		\sum_{N=1}^\infty N^{\max(1,p)} \mathscr{P} ( N(\psi_{j,\bm{G},\bm{k}}) = N )
		& = 
		\sum_{N=1}^\infty N^{\max(1,p)}   \mathrm{e}^{-2^{jd} \lambda \mathrm{Leb}(I_{\psi_{\bm{G}}})} \frac{(\lambda \mathrm{Leb}(I_{\psi_{\bm{G}}}))^N 2^{-jdN}}{N!}  \nonumber \\
				& \leq 
				\left( \sum_{N=1}^\infty N^{\max(1,p)} \frac{(\lambda \mathrm{Leb}(I_{\psi_{\bm{G}}}))^N}{N!}  \right) 2^{-jd},
	\end{align}
	where we used that $2^{-jdN} \leq 2^{-jd}$ and $\mathrm{e}^{-2^{jd} \lambda \mathrm{Leb}(I_{\psi_{\bm{G}}})} \leq 1$.
	We have moreover the relation
	\begin{equation}\label{eq:equation3pourPoisson}
	\lVert \psi_{j,\bm{G},\bm{k}} \rVert_\infty^p = 2^{jpd/2} \lVert \psi_{\bm{G}}\rVert_\infty^p.
	\end{equation}
	Applying  inequalities \eqref{eq:equation2pourPoisson} and \eqref{eq:equation3pourPoisson} in \eqref{eq:equation1pourPoisson} with $\psi = \psi_{j,\bm{G},\bm{k}}$, we finally deduce \eqref{eq:thecontrolPoisson} (for the finite constant $C=\lVert \psi_{\bm{G}}\rVert_\infty^p \mathbb{E} \left[\abs{a_1'}^p\right]   \sum_{N\geq1}N^{\max(1,p)}  (\lambda \mathrm{Leb}(I_{\psi_{\bm{G}}}))^N  / N!$).
	\end{proof}
	
	\begin{proposition} \label{prop:PoissonBesov}
	Fix $0< p \leq \infty$ and $\tau, \rho \in \R$.	
	Let $w$ be a compound Poisson white noise with index $\indmom \in (0,\infty]$.
	If $ 0 < p <  \infty$, then, $w$ is 
		\begin{itemize}
		\item almost surely in $B_p^\reg(\R^d;\rate)$ if $\reg < d/p - d$ and $\rate < - d / \min(p ,\indmom) $,   and
		\item almost surely not in $B_p^\reg(\R^d;\rate)$   if $\reg \geq d/p - d$ or $\rate > -d / \min(p ,\indmom) $, or $\rate \geq - d /p$.
		\end{itemize}	
	If $ p = \infty$, then $w$ is 
		\begin{itemize}
		\item almost surely in $B_\infty^\reg(\R^d;\rate)$ if $\reg < - d$ and $\rate < - d / \indmom$,    and  
		\item almost surely not in $B_\infty^\reg(\R^d;\rate)$  if $\reg > - d$ or $\rate > - d /\indmom$, or $\rho \geq 0$.	
		\end{itemize}
	\end{proposition}

In Proposition \ref{prop:PoissonBesov}, we have split the results in two scenarii depending on what is know when  $\rho$ and/or $\tau$ are equal to the critical values. The only remaining cases that are not covered by Proposition \ref{prop:PoissonBesov} is when $\indmom < \infty$ and $\rho = -d /\indmom$ and when $p = \infty$ and $\tau = -d$. 
As for the Gaussian white noise in Corollary \ref{coro:GaussBesov}, a direct consequence of Proposition \ref{prop:PoissonBesov} is the identification of the local smoothness and the asymptotic growth rate of a compound Poisson white noise.

\begin{corollary} \label{coro:PoissonBesov}
	Let $w$ be a compound Poisson white noise with moment index $0< \indmom \leq \infty$ and $0 < p \leq \infty$. Then, we have almost surely that
	\begin{equation} \label{eq:taurhopoisson}
	\tau_p(w) = \frac{d}{p} - d  \quad \text{ and } \quad \rho_p(w) = - \frac{d}{\min (p,\indmom)}.
	\end{equation}
\end{corollary}

\begin{proof}
The result is easily deduced from the definition of $\tau_p(w)$ and $\rho_p(w)$ in Section \ref{sec:taurho}. Note that the two first cases treated in Proposition \ref{prop:PoissonBesov} gives together that $\tau_p(w) = \frac{d}{p} - d$ and $\rho_p(w) = - \frac{d}{\min(p, \indmom)}$ for $p < \infty$ while the two last ones gives $\tau_\infty(w) = - d$ and $\rho_\infty(w) = - \frac{d}{\indmom}$. Finally, \eqref{eq:taurhopoisson} condenses all the results. 
\end{proof}

\begin{proof}[Proof of Proposition \ref{prop:PoissonBesov}]
	\textbf{Case $p < \indmom$, $\reg < d/p - d$, and $\rate < - d / p$.}
	Under these assumptions, we apply Proposition \ref{lemma:momentPoisson} to deduce that
	\begin{align} \label{eq:intermposCPreg}
		\mathbb{E} \left[ \lVert w \rVert_{B_p^\tau(\R^d;\rate)}^p \right] 
		&= \sum_{j \in \N} 2^{j (\tau p - d + dp/2)} \sum_{\bm{G} \in \G^j} \sum_{\bm{k} \in \Z^d} \langle 2^{-j} \bm{k} \rangle^{\rate p} \mathbb{E} [ \abs{\langle w, \psi_{j,\bm{G},\bm{k}} \rangle}^p ] \nonumber \\
		& \leq 
		C 2^d \sum_{j \in \N} 2^{j (\tau p - d + dp)} \frac{1}{2^{jd}} \sum_{\bm{k} \in \Z^d}  \langle 2^{-j} \bm{k} \rangle^{\rate p},
	\end{align}
	where $C$ is the constant appearing in \eqref{eq:thecontrolPoisson}, and using that	$\mathrm{Card}(\G^j) \leq 2^d$. Then, $$\frac{1}{2^{jd}} \sum_{\bm{k} \in \Z^d}  \langle 2^{-j} \bm{k} \rangle^{\rate p} \underset{j \rightarrow \infty}{\longrightarrow} \int_{\R^d} \langle \bm{x} \rangle^{\rho p}\drm \bm{x} < \infty$$ because $\rate < - d / p$. The sum in \eqref{eq:intermposCPreg} is therefore finite if and only if $\sum_{j} 2^{j (\reg p - d + dp)} < \infty$, which happens here due to our assumption $(\tau p - d + dp)<0$. This shows that $w \in B_p^\tau(\R^d;\rate)$ almost surely. \\
	
\textbf{Case $\indmom \leq p$, $\reg < d/p - d$, and $\rate < - d / \indmom $.}
	We prove that $w \in  B_p^\reg(\R^d;\rate)$ a.s. using the embeddings between Besov spaces and the study of the case $p < \indmom$ before.
	From the conditions on $\reg$ and $\rate$, one can find $p_0 \in (0,\indmom )$, $\tau_0 \in \R$, and $\rho_0 \in \R$ such that
	\begin{align}
		\rate <   \rate_0  < - \frac{d}{p_0} < - \frac{d}{\indmom}, \\
		\tau + \frac{d}{p_0} - \frac{d}{p} <   \reg_0 < \frac{d}{p_0} - d.
	\end{align}
	Then, in particular, $p_0 < p$, $\tau_0 - \tau > d / p_0 - d / p$, and $\rho_0 > \rho$,
	so that $B_{p_0}^{\reg_0}(\R^d;\rate_0) \subset B_p^\reg(\R^d;\rate)$ (according to \eqref{eq:embedd1}).
	Moreover, $p_0 < \indmom$, $\tau_0 < d / p_0 -d$, and $\rate_0 < - d / p_0$. We are therefore back to the first case, for which we have already shown that $w \in B_{p_0}^{\reg_0}(\R^d;\rate_0)$ a.s.
	In conclusion, $w \in  B_p^\reg(\R^d;\rate) $ a.s. \\
	
	Combining these first two cases, we obtain that $w \in  B_p^\reg(\R^d;\rate)$ if $\reg < d/p - d$ and $\rate < - d /\min(p,\indmom)$ for every $p\in (0 ,\infty]$. \\

	\textbf{Case $p < \infty$ and $\reg \geq d/p - d$.} We use the representation \eqref{eq:poissonwithdirac} of the compound Poisson white noise.
	Assume that $w$ is in   $B_p^\reg(\R^d;\rate)$ for some $\rate \in \R$. 
	Then, the product of $w$ by any compactly supported smooth test function $\varphi$ is well-defined and also in $B_p^\reg(\R^d;\rate)$. Choosing a (random) test function $\varphi \in \mathcal{S}(\R^d)$ such that $\varphi(\bm{x}_0) = 1$ and $\varphi(\bm{x}_n) = 0$ for $n \neq 0$, we get 
	\begin{equation}
		\varphi  w = \varphi  \cdot a_0 \delta(\cdot - \bm{x}_0) =  a_0 \delta(\cdot - \bm{x}_0) \in B_p^\reg(\R^d;\rate),
	\end{equation}
	where $a_0 \neq 0$ a.s. This is absurd due to Proposition  \ref{prop:DiracBesov}, proving  that $w \notin  B_p^\reg(\R^d;\rate)$ for all $\rate \in \R$. \\

	\textbf{Case $p = \infty$ and $\reg > - d$.} The same argument than for the case $p < \infty$ and $\reg \geq d/p - d$ applies, using this time that $w \notin B_\infty^{\tau} (\R^d;\rho)$ for any $\rho > 0$, again due to Proposition \ref{prop:DiracBesov}. \\

	\textbf{Case  $\rate \geq - d /p$.} This case has been treated in full generality in Proposition \ref{prop:uneptiteprop}. \\
	
	\textbf{Case  $p > \indmom$ and $\rate > - d / \indmom$.}
	This means in particular that $\indmom < \infty$. 
	We treat the case $\rate < 0$, the extension for $\rate \geq 0$ clearly follows from the embedding relations between Besov spaces.
	We set $q := - d / \rate > \indmom$. In particular, according to Proposition \ref{prop:momentnoisepmax}, we have that $\mathbb{E} [\lvert \langle w , \varphi\rangle \rvert^q] = \infty$ for any compactly supported and bounded function $\varphi \neq 0$. 
	Proceeding as for \eqref{eq:trucgaussianpeugaussienbis}, we have that
\begin{equation} \label{eq:lowerbesovwfather2}
		\lVert w \rVert^p_{B_p^\reg(\R^d;\rate)}  
		\geq 
		\sum_{\bm{k} \in k_0 \Z^d} \frac{\lvert \langle w , \phi_{\bm{k}} \rangle\rvert^p}{\langle \bm{k}\rangle^{dp/q}},
	\end{equation}
	where $k_0 \geq 1$ is chosen such that the functions $\phi_{\bm{k}} =\phi(\cdot - \bm{k})$ have disjoint supports. Then, the random variables $X_{\bm{k}} = \langle w, \phi_{\bm{k}} \rangle$ are i.i.d. The independence implies that the events $A_{\bm{k}} = \{ X_{\bm{k}} \geq \langle \bm{k} \rangle^{d/q} \}$ are independent themselves.
	Then, the $X_{\bm{k}}$ having the same law, we have
	\begin{equation} \label{eq:sumAk}
		\sum_{\bm{k} \in k_0\Z^d} \mathscr{P}( A_{\bm{k}} ) 
		= \sum_{\bm{k} \in k_0\Z^d} \mathscr{P}( |X_{\bm{k}}| \geq \langle \bm{k} \rangle^{d/q} )  =  \sum_{\bm{k} \in k_0\Z^d} \mathscr{P}( \abs{X_{\bm{0}}}^{q} \geq \langle \bm{k} \rangle^d ) \geq \sum_{m \geq 1}\mathscr{P} (\abs{X_{\bm{0}}}^{q} \geq m k_0).
	\end{equation}
	Moreover, we have $
	\mathbb{E} [\abs{X_{\bm{0}}}^{q}]  = \int_0^\infty \mathscr{P}(\abs{X_{\bm{0}}}^{q} \geq x) \drm x  
	 = \sum_{m \geq 1} \int_{m k_0}^{(m+1) k_0} \mathscr{P}( \abs{X_{\bm{0}}}^{q} \geq x) \drm x  $.
	  Exploiting that $\mathscr{P} ( \abs{X_{\bm{0}}}^{q} \geq x )$ is decreasing in $x$, we moreover have that $\int_{m k_0}^{(m+1) k_0} \mathscr{P}( \abs{X_{\bm{0}}}^{q} \geq x) \drm x \leq k_0  \mathscr{P}( \abs{X_{\bm{0}}}^{q} \geq mk_0 )$, and therefore, 
	  \begin{equation} \label{eq:momentY}
	  \mathbb{E} [\abs{X_{\bm{0}}}^{q}]   \leq  \sum_{m \geq 1}  \mathscr{P}( \abs{X_{\bm{0}}}^{q} \geq mk_0 )   .
	  \end{equation} 
	The choice of $q$ implies moreover that  $\mathbb{E} [\abs{X_{\bm{0}}}^{q}] =  \mathbb{E} [\abs{\langle w , \phi\rangle }^{q}] =  \infty$ due to Proposition \ref{prop:momentnoisepmax}.
	Hence, from \eqref{eq:sumAk} and \eqref{eq:momentY}, we deduce that  $	\sum_{\bm{k} \in k_0\Z^d} \mathscr{P}( A_{\bm{k}} ) \geq \mathbb{E}[\abs{X_{\bm{0}}}^{q}]  = \infty$. The Borel-Cantelli lemma  then implies that $ \abs{X_{\bm{k}} }^p \geq \langle \bm{k} \rangle^{  d p / q}$ for infinitely many $\bm{k}$ a.s. Back to \eqref{eq:lowerbesovwfather2}, this implies that $\lVert w \rVert^p_{B_p^\reg(\R^d;\rate)}  = \infty$ a.s. \\	
\end{proof}

\section{Moment Estimates for the  L\'evy White Noise} \label{sec:estimates}

	The proof of Theorem \ref{maintheo:regdecaylimit} will be based on new estimates for the moments $\mathbb{E} [\lvert \langle w , \varphi \rangle\rvert^p]$ of L\'evy white noises. In Section \ref{subsec:higherorder}, we consider the case $p = 2m$ where $m\geq 1$ is an integer. This will be critical when dealing with L\'evy white noises with finite moments. 
	In Section \ref{subsec:theolowerupper}, we determine lower bounds for the moments, {which is the main technicality for the negative Besov regularity results of L\'evy white noises.}

	\subsection{Moment Estimates for $p = 2m \geq 2$} \label{subsec:higherorder}
	
	We estimate the evolution of the even moments of the wavelet coefficients of a L\'evy white noise $w$ with the scale $j$. Most of the moment estimates in the literature deal with $p$th moments with the restriction $p\leq 2$~\cite{Deng2015shift,Fageot2017multidimensional,Kuhn2017existence,Luschgy2008moment}, and we shall see that the extension to higher-order moments calls for some technicalities.

	\begin{proposition} \label{lemma:boundcoefficientfinitemoments}
		Let $w$ be a L\'evy white noise with finite moments and $m\geq 1$ be an integer. 
		We assume that the moment index of $w$ satisfies $\indmom > 2m$.
		Then, there exists a constant $C>0$ such that, for every $j \in \N$, $\bm{G} \in \G^j$, and $\bm{k}\in \Z^d$,
		\begin{equation} \label{eq:boundcoefficientsfinitemoment}
			\mathbb{E} [\langle w, \psi_{j,\bm{G},\bm{k}} \rangle^{2m}] \leq C 2^{j d (m-1)}.
		\end{equation}
	\end{proposition}
	
	\begin{proof}
		Consider a test function $\varphi \in \S(\R^d)$ and set $X = \langle w , \varphi \rangle$.
		The characteristic function of $X$ is~\cite[Proposition 2.12]{Fageotthese}
		\begin{equation} \label{eq:CFX}
			\CF_X(\xi) = \exp\left( \int_{\R^d} \Psi(\xi  \varphi (\bm{x}) )\drm \bm{x} \right) := \exp( \Psi_{\varphi}(\xi)).
		\end{equation}
		The functions $\CF_X$ and $\Psi_{\varphi}$ are $(2m)$-times differentiable because $\mathbb{E}[X^{2m}] < \infty$~{\cite[Theorem 1.5.1]{ushakov2011selected}}. 
		Their Taylor expansions give  the moments and the cumulants of $X$, respectively. In particular, we have that $\mathbb{E}[X^{2m}] = (-1)^m \CF_X^{(2m)}(0)$. 
		Using the Fa\`a di Bruno formula with the composite function $\xi \mapsto \CF_X(\xi) = \exp( \Psi_{\varphi}(\xi))$, we express the $(2m)$th derivative of $\CF_X$ as
		\begin{equation}
			\CF_X^{(2m)}(\xi) = \left( \sum_{n_1,\ldots ,n_{2m} : \sum_u u n_u = 2m} \frac{(2m)!}{n_1 ! \ldots n_{2m}!} \prod_{v = 1}^{2m} \left( \frac{\Psi_{\varphi}^{(v)}(\xi)}{v!}\right)^{n_v}\right) \CF_X(\xi).
		\end{equation}
		Exploiting that $\Psi^{(v)}_\varphi (0) = \left( \int_{\R^d} (\varphi(\bm{x}))^v \drm \bm{x}\right) \Psi^{(v)}(0)$ for $\xi= 0$~\cite[Proposition 9.11]{Unser2014sparse}, we obtain the bound,
		\begin{equation} \label{eq:boundCFXvarphi}
			\abs{\CF_X^{(2m)}(0)} \leq C'  \sum_{n_1,\ldots ,n_{2m} : \sum_u u n_u = 2m} \prod_{v=1}^{2m} \left(\int_{\R^d} \abs{\varphi(\bm{x})}^v \drm \bm{x}\right)^{n_v}		
		\end{equation}
with $C'>0$ a constant. 
We now apply \eqref{eq:boundCFXvarphi} to $\varphi = \psi_{j,\bm{G},\bm{k}}$. 
Since we have
	\begin{equation}
		\int_{\R^d}  \abs{\psi_{j,\bm{G},\bm{k}}(\bm{x})}^v \drm \bm{x} = 2^{j d v / 2} \int_{\R^d} \abs{\psi_{\bm{G}}(2^j \bm{x} - \bm{k})}^v \drm \bm{x} = 2^{jd ( v/2 - 1)} \int_{\R^d} \abs{ \psi_{\bm{G}}(\bm{x})}^v \drm \bm{x},
	\end{equation}
	we deduce from \eqref{eq:boundCFXvarphi} the new bound
	\begin{align}
	\mathbb{E} [\langle w, \psi_{j,\bm{G},\bm{k}} \rangle^{2m}] &= \abs{\CF_{\langle w, \psi_{j,\bm{G},\bm{k}}\rangle}^{(2m)} (0) }   \leq  C''  \sum_{n_1,\ldots , n_{2m} : \sum_u u n_u = 2m} \prod_{v=1}^{2m} 
						2^{jdn_v ( v/2 - 1)} \nonumber \\
				& = C'' \sum_{n_1,\ldots , n_{2m} : \sum_u u n_u = 2m} 2^{j d  \sum_{1\leq v \leq 2m} \left(n_v ( v/2 - 1)\right)},
	\end{align}	
	where $C''$ is a new constant independent from $j, G, \bm{k}$. Finally, since $\sum_v n_v  v = 2m$ and $\sum_v n_v \geq 1$, we have $ \sum_v \left(n_v ( v/2 - 1)\right) \leq (m -1)$. Therefore, 
	we obtain \eqref{eq:boundcoefficientsfinitemoment}
	for an adequate $C>0$.
	\end{proof}

	\subsection{Lower Bound for Moment Estimates} \label{subsec:theolowerupper}

In our previous moment estimates, we gave upper bounds for the quantity $\mathbb{E}[ \abs{\langle w, \varphi\rangle}^p]$ (see not only Propositions \ref{lemma:momentPoisson} and \ref{lemma:boundcoefficientfinitemoments}, but also Theorem 2 in \cite{Fageot2017multidimensional}). This allows one  to identify in which Besov space is $w$. We now address the following problem: Can we  bound $\mathbb{E}[ \abs{\langle w, \varphi\rangle}^p]$ from below with the moments of $\varphi$? Theorem \ref{lemma:lowerbound} answers positively to this question and is crucial for the proof of Theorem \ref{maintheo:regdecaylimit}.

\begin{theorem} \label{lemma:lowerbound}
		Let $w$ be a L\'evy white noise whose  indices satisfy $0 < \indlocinf \leq \indlocsup  < \indmom$,
		$\psi \neq 0$ a bounded, and compactly supported test function, 
		and $p$ an integrability parameter such that $0<p< \indlocsup$.
		Then, for $\epsilon>0$ small enough, there exists constants $A, B > 0$ independent from $j\in \N$ and $\bm{k}\in\Z^d$ such that
		\begin{equation} \label{eq:boundsonmoments}
			A 2^{- j\epsilon} 2^{j d p (1 / 2 - 1 / \indlocinf)} \leq \mathbb{E} [ \abs{\langle w , \psi_{j,\bm{k}} \rangle}^p] \leq B 2^{j\epsilon}  2^{j d p (1 / 2 - 1 / \indlocsup)} 
		\end{equation}		
		for any $j \in \N$ and $\bm{k} \in \Z^d$, where we recall that $ \psi_{j,\bm{k}} = 2^{jd/2} \psi(2^j \cdot -\bm{k})$.
	\end{theorem}

	\begin{proof}
		First of all, the shift parameter $\bm{k}$ in \eqref{eq:boundsonmoments} can be omitted since $w$ is stationary. We also remark that the upper bound of \eqref{eq:boundsonmoments} has already been proven~\cite[Corollary 1]{Fageot2017multidimensional}, where the conditions $p<\indlocsup < \indmom$ are required. Actually, \cite{Fageot2017multidimensional} does not consider the index $\indmom$ and distinguishes between the conditions $\indlocsup < \indasysup<2$ and $\indlocsup \leq \indasysup = 2$ with finite variance. These two scenarios cover the condition $\indlocsup < \indmom$ of Theorem \ref{lemma:lowerbound}. 		
		Hence, we focus on the lower bound.
		
		Because $p<\indlocsup \leq 2$, one can use the  representation of the $p$th moment of $\langle w, \varphi\rangle$, that can be found in {\cite[Theorem 1.5.9]{ushakov2011selected} (with $n = 0$ and $\delta = p$)}:
		\begin{equation}\label{eq:pthmomentRV}
		\mathbb{E}[ \abs{\langle w, \varphi\rangle}^p] = c_p \int_{\R} \frac{1 - \Re \{ \CF_{\langle w, \varphi\rangle} (\xi  ) \}  }{\abs{\xi}^{p+1}} \mathrm{d}\xi
		\end{equation}
		for some explicit constant $c_p>0$.The relation \eqref{eq:pthmomentRV} is often used for moment estimates, for instance in \cite{Deng2015shift,Luschgy2008moment}.
	We then remark that
	\begin{align}
		 \Re \{ \CF_{\langle w, \varphi\rangle} (\xi) \} 
		& \leq \abs{\CF_{\langle w, \varphi\rangle} (\xi)} = \abs{\CF_w(\xi \varphi)} = \abs{\exp\left(\int_{\R^d} \Psi(\xi\varphi(\bm{x})) \mathrm{d}\bm{x} \right) } \nonumber \\
		 &= \exp\left( \int_{\R^d} \Re \{\Psi (\xi \varphi(\bm{x})) \} \mathrm{d}\bm{x} \right).
	\end{align}
	The test function $\varphi$ is chosed to be non-identically zero, hence there exists some constant $m >0$ such that $\mathrm{Leb}(\abs{\varphi} \geq m \lVert \varphi \rVert_\infty) > 0$. We fix such a constant $m$ and observe that 
	\begin{equation}
	\label{eq:withthenewm}
	\varphi \One_{\abs{\varphi} >  m \lVert \varphi \rVert_\infty } \neq 0.
	\end{equation}
	The sector condition \eqref{eq:sectorcondition} implies that one can found $c_{\Psi}>0$ such that $- \Re \{\Psi(\xi)\} =  \abs{\Re \{ \Psi \}} (\xi)  \geq c_{\Psi} \abs{\Psi(\xi)}$. Thus, one has that $\abs{\CF_{\langle w, \varphi\rangle} (\xi)} \leq \exp\left( -c_{\Psi} \int_{\R^d} \abs{\Psi(\xi \varphi(\bm{x})) }\mathrm{d}\bm{x} \right)$, and then	
	\begin{equation} \label{eq:lowerbound1}
		\mathbb{E}[ \abs{\langle w, \varphi\rangle}^p]  \geq   c_p \int_{\R} \frac{1 - \mathrm{e}^{ -c_{\Psi} \int_{\R^d} \abs{\Psi(\xi \varphi(\bm{x})) }\mathrm{d}\bm{x}} }{\abs{\xi}^{p+1}} \mathrm{d}\xi.
	\end{equation}
	Now, by definition of the index $\indlocinf$, the function $\abs{\Psi}$ is dominating 
	$\abs{\xi}^{\indlocinf - \delta}$ at infinity 
	for an arbitrarily small $\delta$ such that $0< \delta < \indlocinf$. We fix such $\delta$ and set $\beta = (\indlocinf - \delta)$. 
	This domination, together with the continuity of the functions $\abs{\Psi}$ and $\abs{\cdot}^{\beta} \One_{\abs{\cdot}>1}$ over $[1,\infty)$,  imply  the existence of $C>0$ such that, for any $\xi \in \R$, $\abs{\Psi(\xi)} \geq C \abs{\xi}^{\beta} \One_{\abs{\xi}>1}$. In particular, we have that
	\begin{align} \label{eq:newbigbound}
		\int \abs{\Psi(\xi \varphi(\bm{x}))} \mathrm{d}\bm{x} &\geq C \abs{\xi}^{\beta} \int_{\R^d} \abs{\varphi(\bm{x})}^{\beta} \One_{\abs{\xi \varphi(\bm{x})}>1} \mathrm{d}\bm{x}  \nonumber \\
		& \overset{(i)}{\geq} C \One_{\abs{\xi } > 1 /  m \lVert \varphi \rVert_\infty} 
		\abs{\xi}^{\beta} \int_{\R^d} \abs{\varphi(\bm{x})}^{\beta} \One_{\abs{\varphi(\bm{x})}>m \lVert \varphi \rVert_\infty} \mathrm{d}\bm{x} \nonumber \\
		& = C \One_{\abs{\xi } > 1 / m \lVert \varphi \rVert_\infty} 
		\abs{\xi}^{\beta} \lVert \varphi \One_{\abs{\varphi} > m \lVert \varphi \rVert_\infty } \rVert_\beta^\beta,
	\end{align}
	where the $(i)$  uses that 
$	\One_{\abs{\xi \varphi(\bm{x})}>1} 
	 \geq \One_{\abs{\varphi(\bm{x})}> m \lVert \varphi \rVert_\infty }\One_{\abs{\xi } > 1/ m \lVert \varphi \rVert_\infty} $, where $m$ is such that \eqref{eq:withthenewm} holds. 
	 
	Combining \eqref{eq:lowerbound1} and \eqref{eq:newbigbound}, we therefore have
	 \begin{align}
	 \mathbb{E}[ \abs{\langle w, \varphi\rangle}^p] 
	 \geq
	  c_p \int_{\R} \frac{1 - \mathrm{e}^{ -c_{\Psi} C \One_{\abs{\xi } > 1/ m \lVert \varphi \rVert_\infty} 
		\abs{\xi}^{\beta} \lVert \varphi \One_{\abs{\varphi} > m \lVert \varphi \rVert_\infty } \rVert_\beta^\beta} }{\abs{\xi}^{p+1}} \mathrm{d}\xi.
	 \end{align}
	 We use the change of variable $u = (c_\Psi C)^{1/\beta} \lVert \varphi \One_{\abs{\varphi} > m  \lVert \varphi \rVert_\infty  } \rVert_\beta \xi$ to obtain
	 \begin{align} \label{eq:encoreuneroger}
	 	 \mathbb{E}[ \abs{\langle w, \varphi\rangle}^p] 
	 	 & \geq C' \lVert \varphi \One_{\abs{\varphi} > m \lVert \varphi \rVert_\infty  } \rVert_\beta^p \int_{\R} \frac{1 - \mathrm{exp}\left(- \abs{u}^\beta \One_{\abs{u} > C'' \frac{\lVert \varphi \One_{\abs{\varphi} > m \lVert \varphi \rVert_\infty } \rVert_\beta}{\lVert \varphi \rVert_\infty}}\right) }{\abs{u}^{p+1}} \mathrm{d}u
	 \end{align}
	 for some constants $C',C''>0$.We now set $\varphi = \psi_j = 2^{jd/2} \psi(2^{j} \cdot)$, and observe that, with simple changes of variable,
	\begin{align}
		\lVert \varphi \One_{\abs{\varphi} > m \lVert \varphi \rVert_\infty } \rVert_\beta^p & = 2^{j d p \left(\frac{1}{2} - \frac{1}{\beta}\right)} \lVert \psi \One_{\abs{\psi} >m  \lVert \psi \rVert_\infty} \rVert_\beta^p, \label{eq:trutrutruc} \\
		\frac{\lVert \varphi \One_{\abs{\varphi} > m \lVert \varphi \rVert_\infty } \rVert_\beta}{\lVert \varphi \rVert_\infty}
		&= \frac{2^{j d \left( \frac{1}{2} - \frac{1}{\beta}\right)} \lVert \psi \One_{\abs{\psi} >m  \lVert \psi \rVert_\infty} \rVert_\beta}{2^{j d /2}\lVert \psi \rVert_\infty}= 2^{- j d / \beta} \frac{ \lVert \psi \One_{\abs{\psi} > m \lVert \psi \rVert_\infty } \rVert_\beta}{\lVert \psi \rVert_\infty}. \label{eq:ratioofnormpsi}
	\end{align}
	In particular, 
	\begin{equation}
	\One_{\abs{u} > C'' 2^{- j d / \beta} \frac{ \lVert \psi \One_{\abs{\psi} >m  \lVert \psi \rVert_\infty} \rVert_\beta}{\lVert \psi \rVert_\infty}} \geq \One_{\abs{u} > C'' \frac{ \lVert \psi \One_{\abs{\psi} >m  \lVert \psi \rVert_\infty} \rVert_\beta}{\lVert \psi \rVert_\infty}}	
	\end{equation}		
	for any $j \in \N$. Hence, we deduce using \eqref{eq:encoreuneroger} with $\varphi = \psi_j$ that
	\begin{equation}
		 	 \mathbb{E}[ \abs{\langle w, \psi_j \rangle}^p] 
	\geq 
	B 2^{j d p \left( \frac{1}{2} - \frac{1}{\beta}\right)}
	\end{equation}
	with $B >0$ a constant given by
	\begin{equation}
	B = C' 
	\lVert \psi \One_{\abs{\psi} >m  \lVert \psi \rVert_\infty} \rVert_\beta^p
	\int_{\R} 
	\frac{1 - 
	\mathrm{exp}\left(
	- \abs{u}^\beta 
	 \One_{\abs{u} > C'' \frac{ \lVert \psi \One_{\abs{\psi} >m  \lVert \psi \rVert_\infty} \rVert_\beta}{\lVert \psi \rVert_\infty}}
	\right) }{\abs{u}^{p+1}} \mathrm{d}u.
	\end{equation}
	Remark that $B \neq 0$ because $\psi \One_{\abs{\psi} >m  \lVert \psi \rVert_\infty} \neq 0$ due to \eqref{eq:trutrutruc} and   \eqref{eq:withthenewm} .
		To conclude, we remark that, for $\epsilon >0$ fixed, one can find $\delta>0$ small enough such that $2^{-j\epsilon} 2^{j d  \left( \frac{1}{2}- \frac{1}{\indlocinf}\right)} \leq2^{j d  \left( \frac{1}{2}- \frac{1}{\indlocinf - \delta}\right)} = 2^{jd \left(\frac{1}{2}- \frac{1}{\beta}\right)}$ for any $j \in \N$, which gives the lower bound in \eqref{eq:boundsonmoments}. 
	\end{proof}

\section{L\'evy White Noise with Finite Moments} \label{sec:finitemoments}

We consider L\'evy white noises whose all the moments are finite, which means that $\indmom = \infty$. 
Their specificity is that one can use the finiteness of the $p$th moments of the wavelet coefficients of the L\'evy white noise $w$ for any $p>0$.
Thanks to the moment estimates in Section \ref{sec:estimates}, we have all the tools to deduce the Besov regularity of  white noises {with finite moments.}
	
		\begin{proposition} \label{prop:finiteMomentsBesov}
	Fix $0< p \leq \infty$ and $\tau, \rho \in \R$.	
	Let $w$ be a  L\'evy white noise with finite moments and Blumenthal-Getoor indices $0 \leq \indlocinf \leq \indlocsup \leq 2$.
		Then, $w$ is 
\begin{itemize}
	\item almost surely in $B_p^\reg(\R^d;\rate)$ if $\reg < d/\max(p,\indlocsup) - d$ and $\rate < - d /p$, for $0<p\leq 2$, $p$ an even integer, or $p=\infty$; and
	\item  almost surely not in $B_p^\reg(\R^d;\rate)$ if $\reg >  d/\max(p,\indlocinf) - d$ or $\rate \geq  - d /p$ for every $0<p \leq \infty$.
\end{itemize}
	\end{proposition}
	
  	As was the case for the Gaussian and compound Poisson cases, Proposition \ref{prop:finiteMomentsBesov} allows to deduce the asymptotic growth rate of L\'evy white noises with finite moments. Moreover, we obtain lower and upper bounds for the local smoothness in terms of the Blumenthal-Getoor indices of the L\'evy white noise.
	
	\begin{corollary} \label{coro:finiteMomentsBesov}
	Let $0< p \leq \infty$ and $w$ be a L\'evy white noise with finite moments and Blumenthal-Getoor indices $0 \leq \indlocinf \leq \indlocsup \leq 2$. Then, we have that
	\begin{equation} \label{eq:taurhofinitemom}
	\frac{d}{\max(p,\indlocsup)} - d  \leq \tau_p(w) \leq \frac{d}{\max(p,\indlocinf)} - d  \quad \text{ and } \quad \rho_p(w) \geq - \frac{d}{ p}.
	\end{equation}
	For $p\in (0,2)$, $p$   an even integer, or $p = \infty$, we moreover have that 
		\begin{equation} 
 \rate_p(w) = - \frac{d}{p}.
 	\end{equation}	
	\end{corollary}
	 
	\begin{proof}[Proof of Proposition \ref{prop:finiteMomentsBesov}]
	We only treat the case $p< \infty$. For $p=\infty$, the result is obtained using embeddings (with $p=2m$, $m\rightarrow \infty$ for positive results and $p\rightarrow \infty$ for negative results) following the same arguments than for the Gaussian case in Proposition \ref{prop:GaussBesov}.\\
	
	\textbf{ If $\reg < d/\max(p,\indlocsup) - d$ and $\rate < - d /p$.}
	We first remark that for $p \leq 2$, we have that $\rate < - \frac{d}{\min(p,2, \indmom)} =   - \frac{d}{\min(p, \indmom)}$, and the result is a consequence of our previous work \cite[Theorem 3]{Fageot2017multidimensional}. We can therefore assume that $p = 2m$  with $m \geq 1$ an integer. Then, $p=2m \geq 2 \geq \indlocsup$, and the conditions on $\reg$ and $\rate$ become  $\reg < \frac{d}{2m} - d$ and $\rate < - \frac{d}{2m}$.
	Due to Proposition \ref{lemma:boundcoefficientfinitemoments}, we have
	\begin{align} \label{eq:seriestobefinite}
		\mathbb{E} \left[\lVert w \rVert_{B_{2m}^\reg(\R^d;\rate)}^{2m}  \right]
			&= \sum_{j \in \N} 2^{j (2m\reg - d + dm)}\sum_{\bm{G} \in \G^j}
		\sum_{\bm{k}\in \Z^d}  \langle 2^{-j}\bm{k} \rangle^{2m\rate} \mathbb{E} [\langle w, \psi_{j,\bm{G},\bm{k}} \rangle^{2m}] \nonumber \\
			& \leq C 2^d \sum_{j \in \N} 2^{j (2m\reg - d + 2dm)} \frac{1}{2^{jd}} \sum_{\bm{k} \in \Z^d}  \langle 2^{-j}\bm{k} \rangle^{2m\rate}. 
	\end{align}
	Then, 
	$ \frac{1}{2^{jd}} \sum_{\bm{k} \in \Z^d}  \langle 2^{-j}\bm{k} \rangle^{2m\rate} \underset{j \rightarrow \infty}{\longrightarrow} \int_{\R^d} \langle \bm{x} \rangle^{2 m \rate} \drm \bm{x} < \infty$ 
	since $2 m \rate < - d$, 
	and the series in \eqref{eq:seriestobefinite} is finite if and only if 
	$2 m \reg - d + 2dm < 0$, what we assumed to be true.
	Finally, we have shown that $\mathbb{E} \left[\lVert w \rVert_{B_{2m}^\reg(\R^d;\rate)}^{2m} \right] < \infty$, so that $w \in B_{2m}^\reg(\R^d;\rate)$ almost surely.  \\
	
	\textbf{Case $\reg > d / p - d$.}
	This part of the proof is actually valid for any L\'evy white noise. It uses the decomposition $w =  w_1 + w_2$ with $w_1$ a nontrivial compound Poisson white noise and $w_2$ a L\'evy white noise with finite moments (see Proposition \ref{prop:LIdecompo}, here $w_2$ combines the Gaussian part and the finite-moment part of the L\'evy-It\^o decomposition). The main idea is that the jumps of the compound Poisson part are by themselves enough to make the Besov norm infinite. 

	One writes that $w_1 \overset{(\mathcal{L})}{=}  \sum_{k\in \mathbb{Z}} a_k \delta(\cdot - \bm{x}_k)$, as in \eqref{eq:poissonwithdirac}. Then, almost surely, all the $a_k$ are nonzero. Let $m>0$ be such that $\mathscr{P} (\abs{a_0} \geq m) > 0$. One can assume without loss of generality that $m=1$ (otherwise, consider the white noise $w/m$). Then, there exists almost surely $k \in \Z$ such that $\abs{a_k} \geq 1$ (it suffices to apply the Borel-Cantelli argument using the independence of the $a_k$). For simplicity, one reorders the jumps such that this is achieved for $k=0$, so that  $\abs{a_0} \geq 1$.

	We first introduce preliminary notations.
	Let $[a,b]$ be a finite interval on which the mother Daubechies wavelet is strictly positive. In particular, $\min_{x \in [a,b)} \lvert \psi_M(x) \rvert := c >0$. Then, setting $K=[a,b)^d$ and $C = c^d$,  we have that $\abs{\psi_{\bm{M}} (\bm{x}) } \geq C$ for every $\bm{x} \in K$, with $\bm{M} = (M,\ldots ,M)$.
	For each scale $j \geq 0$, we define $\bm{k}_j \in \Z^d$ as the unique multi-integer such that $2^j \bm{x}_0 - \bm{k}_j \in [a, a +1)^d$. 
	If $b \geq a + 1$, then $2^j \bm{x}_0 - \bm{k}_j \in K$. Otherwise, due to the law of the location $\bm{x}_0$, there is almost surely an infinity of scale $j \geq 0$ such that $2^j \bm{x}_0 - \bm{k}_j \in [a, b)^d = K$. We denote by $J$ the (random) ensemble of such $j$.
	Then, for each $j \in J$, using that $\lvert a_0 \rvert \geq 1$ and $2^j \bm{x}_0 - \bm{k}_j \in [a, b)^d = K$, we deduce that
	\begin{equation}
	\abs{ a_0 \psi_{\bm{M}}(2^j \bm{x}_0 - \bm{k}_j ) }
			\geq C.
			\end{equation}
	
	Moreover, on each finite interval, there are almost surely finitely many jumps $\bm{x}_k$.
	In particular, 
	the random variable $\inf_{k \in \Z \backslash \{0\} } \lVert \bm{x}_k - \bm{x}_0 \rVert $ is a.s. strictly positive. 
	This implies that
 	there exists a (random) integer  $j_0 \in \N$   such that $2^{j_0} \lVert \bm{x}_k - \bm{x}_0 \rVert  >   \mathrm{diam} ( \mathrm{Supp}( \psi_{\bm{M}} ))$ for any $k \neq 0$, where $\mathrm{diam}(B)$ is the diameter of a Borelian set $B \subset \R^d$, understood as the Lebesgue measure of its closed convex hull.
	We therefore have that $\psi_{\bm{M}}(2^j \bm{x}_k - \bm{k}_j) = 0$ for any $j \geq j_0$ and any $k \neq 0$.
	From these preparatory considerations, one has  a.s. that, for $j \geq j_0$, $j\in J$, 
	\begin{equation} \label{eq:lowerboundpoissonpart}
		\abs{ \langle w_1 , \psi_{\bm{M}}(2^j \cdot - \bm{k}_j ) \rangle}
			= \abs{\sum_{k\in \Z} a_k \psi_{\bm{M}}(2^j \bm{x}_k - \bm{k}_j ) }
			= \abs{ a_0 \psi_{\bm{M}}(2^j \bm{x}_0 - \bm{k}_j ) }
			\geq C.
	\end{equation}	
	
	Let now focus on the L\'evy white noise $w_2$. Since $w_2$ has a finite variance, we have that, using the Markov inequality,
	\begin{equation}
		\mathscr{P} 
		\left( \abs{\langle w_2 , \psi_{\bm{M}}(2^j \cdot - \bm{k}_j) \rangle } \geq \frac{C}{2} \right) 
		\leq \frac{4 \mathbb{E} \left[\langle w_2 , \psi_{\bm{M}}(2^j \cdot - \bm{k}_j) \rangle^2 \right]}{C^2} = \frac{4 \sigma_0^2 \lVert \psi_{\bm{M}} \rVert_2^2}{C^2} 2^{-jd},
	\end{equation}
	where $\sigma_0^2$ is the variance of $w_2$ such that $\mathbb{E}[\langle w_2 , \varphi \rangle^2 ] = \sigma_0^2 \lVert \varphi \rVert_2^2$ for any test function $\varphi$. In particular, $$\sum_{j \in \N} \mathscr{P} \left( \abs{\langle w_2 , \psi_{\bm{M}}(2^j \cdot - \bm{k}_j) \rangle } \geq \frac{C}{2} \right) \leq   \frac{4 \sigma_0^2 \lVert \psi_{\bm{M}} \rVert_2^2}{C^2}  \sum_{j\in \N} 2^{-jd}< \infty.$$ 
	From a new Borel-Cantelli argument, we know that, almost surely, only finitely many $j$ satisfy $\abs{\langle w_2 , \psi_{\bm{M}}(2^j \cdot - \bm{k}_j) \rangle } \geq {C} / {2}$. In particular, there exists a (random) integer $j_1$ such that,  for any $j\geq j_1$, $\abs{\langle w_2 , \psi_{\bm{M}}(2^j \cdot - \bm{k}_j) \rangle } \leq {C} / {2}$. Combining this to \eqref{eq:lowerboundpoissonpart}, we then have that, for  $j \in J$ such that $j \geq \max(j_0,j_1)$,
	\begin{align} \label{eq:etvoila}
	\abs{\langle w , \psi_{\bm{M}} (2^{j} \cdot -  \bm{k}_j )\rangle }
	& \geq \abs{\langle w_1 , \psi_{\bm{M}} (2^{j} \cdot -  \bm{k}_j )\rangle } - \abs{\langle w_2 , \psi_{\bm{M}} (2^{j} \cdot -  \bm{k}_j )\rangle }   \geq C - C/2 = C/2.
	\end{align}
	Note moreover that, by definition of $\bm{k}_j$, $2^j \bm{x}_0 - \bm{k}_j \in [a,a+1)^d$, hence we have that $\lVert \bm{x}_0 - 2^{-j} \bm{k}_j \rVert_\infty \leq M/2^j$ where $M = \max ( \abs{a}, \abs{a+1})$.  Then, there exists a (random) integer $j_2 \geq 1$   such that, for every $j \geq j_2$, $M/2^j \leq \lVert \bm{x}_0\rVert_\infty$.	
	We have moreover that $\lVert 2^{-j} \bm{k}_j \rVert_2 \leq d^{1/2} \lVert 2^{-j} \bm{k}_j \rVert_\infty \leq d^{1/2} (M / 2^j + \lVert \bm{x}_0 \rVert_\infty)$. 
	Therefore, recalling that $\rho < 0$, for every $j \geq j_2$, we have
	\begin{equation} \label{eq:encoreune}
		\langle 2^{-j} \bm{k}_j \rangle^{\rho p} \geq (1 + d (M/2^j+ \lVert \bm{x}_0 \rVert_\infty)^2)^{\rho p /2} \geq  (1 + 4d  \lVert \bm{x}_0 \rVert_\infty^2)^{\rho p /2}.
	\end{equation}
	Putting the pieces together, we can now lower bound the Besov norm of $w$ by keeping only the mother wavelet $\psi_{\bm{M}}$, a scale $j \in J$ such that $j \geq \max(j_0, j_1,j_2)$, and the corresponding shift parameter $ \bm{k}_j$. 
		Then, combining \eqref{eq:etvoila} and \eqref{eq:encoreune}, we obtain the almost-sure lower bound
	\begin{align}
		\lVert w \rVert_{B_p^\tau(\R^d;\rho)}^p 
		&\geq 
		2^{j (\tau p - d + dp)} \langle 2^{-j} \bm{k}_j \rangle^{p \rho} \abs{\langle w , \psi_{\bm{M}}(2^{j} \cdot -  \bm{k}_j) \rangle }^p \nonumber \\
	& \geq   2^{j (\tau p - d + dp)} \left(C/2 \right)^p (1 + 4d  \lVert \bm{x}_0 \rVert_\infty^2)^{\rho p /2}.
	\end{align} 
	This is valid for any $j \in J$ such that $j \geq 	\max(j_0, j_1,j_2)$ and because $J$ is infinite and $(\tau p - d + dp) > 0$, one concludes that $\lVert w \rVert_{B_p^\tau(\R^d;\rho)}^p = \infty$ almost surely.\\

	\textbf{Case $0 < p < \indlocinf$ and $( d /  \indlocinf     - d) < \tau < (d/p - d)$.} Assume that, under those assumptions, we prove that $w \notin  B_p^{\tau}(\R^d;\rho)$ a.s. Then, together the case $\tau > (d/p - d)$ considered below and using embeddings, we deduce the expected result for $\tau > p(d / \max( \indlocinf , p) - d)$.
	
	As soon as $f \notin B_p^{\tau}(\R^d;\rho)$ for some $p>0$, we also have that $f \notin B_q^{\tau + \epsilon}(\R^d;\rho)$ for any $q>p$ and $\epsilon > 0$ (see Figure \ref{fig:embeddings}). A crucial consequence for us is that it suffices to work with arbitrarily small $p$ in order to obtain the negative result we expect. We assume here that
	\begin{equation} \label{eq:troisconditionsurp}
	 p < \indlocinf / 2, \quad p < \frac{\indlocinf \indlocsup}{2(\indlocsup - \indlocinf)}.
	\end{equation}
	Note that the right inequality in \eqref{eq:troisconditionsurp}   simply means that $p< \infty$ (\emph{i,e.}, no restriction) when $\indlocinf = \indlocsup$. 
	We fix  $k_0 \in \N \backslash\{0\}$ such that, for any gender $\bm{G}$, the functions $\Psi_{0, G, k_0\bm{k}}$ have disjoint support for every $\bm{k} \in \Z^d$. Then, at fixed $\bm{G}$ and $j$, the  random variables $(\langle w , \Psi_{j, G, \bm{k}}\rangle)_{\bm{k} \in k_0 \Z^d}$ are independent.
	By restricting the range of $\bm{k}$ and the gender to $\bm{G} = \bm{M}$, we have that
	\begin{equation} \label{eq:lowernormbesov}
		\lVert w \rVert_{B_p^\tau(\R^d;\rho)}^p \geq  C \sum_{j \in \N} 2^{j ( \tau p - d + dp /2)} \sum_{\bm{k}\in k_0\Z^d, 0 \leq k_i < k_0 2^j } \abs{\langle w,  \psi_{j,\bm{M},\bm{k}} \rangle}^p,
	\end{equation}
	with $C = \inf_{\lVert \bm{x}\rVert_\infty \leq k_0} \langle \bm{x} \rangle^{\rate} > 0$ is such that   $\langle 2^{-j} \bm{k} \rangle \geq C$ for any $\bm{k} \in  k_0 \{ 0, \ldots  2^j-1\}^d$ and any $j \geq 0$.
	We set $X_{j,\bm{k}} = 2^{jd \left( \frac{1}{\indlocinf} - \frac{1}{2}\right)}    {\langle w, \psi_{j,\bm{M},\bm{k}}\rangle}$ and
	$$M_{j,p} := 2^{-jd} \sum_{\bm{k}\in k_0\Z^d, 0 \leq k_i < k_0 2^j} \abs{X_{j,\bm{k}}}^p,$$
	which is an average among $2^{jd}$ independent random variables. 
	
	Recall that $p <   \indlocinf / 2$. Moreover, since all the moments are finite, $\indmom = \infty > \indlocsup$. Hence, one can apply Theorem \ref{lemma:lowerbound} with integrability parameters  $q = p$ and $q = 2p$, respectively. There exists $\epsilon>0$ that can be choosen arbitrarily small and constants $m_q, M_q$ such that
	\begin{equation}
	m_q 2^{-j \epsilon} 2^{jqd \left(\frac{1}{2} - \frac{1}{\indlocinf}\right)}
	\leq 
	\mathbb{E} \left[\abs{\langle w , \psi_{j,\bm{M},\bm{k}}\rangle}^q\right]
	\leq 
	M_q 2^{j \epsilon} 2^{jqd \left(\frac{1}{2} - \frac{1}{\indlocsup}\right)}
	\end{equation}
	for any $j \in \N, \bm{k}\in \Z^d$.
	In particular, with our notations, we have that
	\begin{equation} \label{eq:boundtoprove}
		m_q 2^{-j \epsilon} \leq \mathbb{E} [M_{j,q}] \leq 
			M_q 2^{j \epsilon} 2^{jqd \left(\frac{1}{\indlocinf} - \frac{1}{\indlocsup}\right)}
	\end{equation}
	for any $j, \bm{k}$ and for $q=p$ or $q=2p$.
	Then, we control the variance of $M_{j,q}$ as follows:
	\begin{align}\label{eq:boundtouse}
		\mathrm{Var}(M_{j,p}) 
		&= 
		\mathbb{E} \left[(M_{j,p} - \mathbb{E}[M_{j,p}])^2  \right]
		\overset{(i)}{=}
2^{-jd} \mathbb{E} \left[2^{-jd}  \left( \sum_{\bm{k}\in k_0\Z^d, 0 \leq k_i < k_0 2^j} \left(\abs{X_{j,\bm{k}}}^p - \mathbb{E}[\abs{X_{j,\bm{k}}}^p] \right) \right)^2 \right] \nonumber  \\
		&\overset{(ii)}{=}
		 2^{-jd} \mathbb{E} \left[2^{-jd}   \sum_{\bm{k}\in k_0\Z^d, 0 \leq k_i < k_0 2^j} \left( (\abs{X_{j,\bm{k}}}^p - \mathbb{E}[\abs{X_{j,\bm{k}}}^p]) \right)^2 \right] \nonumber  \\
		& \overset{(iii)}{\leq}		 
		 2^{-jd} \mathbb{E} \left[2^{-jd}   \sum_{\bm{k}\in k_0\Z^d, 0 \leq k_i < k_0 2^j}  \abs{X_{j,\bm{k}}}^{2p} \right]  = 2^{-jd} \mathbb{E}[M_{j,2p}] \nonumber \\
		 & \overset{(iv)}{\leq} 
		 2^{-jd} 2^{j\epsilon} 2^{2jpd \left(\frac{1}{\indlocinf} - \frac{1}{\indlocsup}\right)},
	\end{align}
	where we used that $	 \mathbb{E}[M_{j,p}] =   \mathbb{E}[ \abs{X_{j,\bm{k}}}^p ]$ for every $\bm{k}$ in $(i)$, the independence of the $X_{j,\bm{k}}$ in $(ii)$,   the relation $\mathrm{Var}(X) \leq \mathbb{E} [X^2]$ in $(iii)$, and the right side of \eqref{eq:boundtoprove} with $q = 2p$ in $(iv)$. 
	
	We then apply the Chebyshev's inequality $\mathscr{P}( \abs{X- \mathbb{E}[X]} \geq x )\leq \mathrm{Var} X / x^2$ to $X=M_{j,p}$ and $x = 2^{-j\epsilon} m_p / 2$ to get
	\begin{align} \label{eq:trucmachin}
	\mathscr{P}( \abs{M_{j,p}- \mathbb{E}[M_{j,p}]} \geq 2^{-j\epsilon} m_p / 2 )
	&\leq \frac{4 \mathrm{Var}(M_{j,p})}{m_p^2} 2^{2j \epsilon}  \nonumber \\
	&	\leq 
	\frac{4 M_{2p}}{m_p^2} 2^{j \left( d ( 2p ( 1 / \indlocinf - 1 / \indlocsup) - 1 ) +  3\epsilon \right)}
	\end{align}
	where we used \eqref{eq:boundtouse} in the last inequality. 
	Due to the second inequality in \eqref{eq:troisconditionsurp}, we have that
	\begin{equation}
	2p \left( \frac{1}{\indlocinf} - \frac{1}{\indlocsup} \right) - 1 <0.
	\end{equation}	
	Hence, the exponent $d ( 2p ( 1 / \indlocinf - 1 / \indlocsup) - 1 ) +  3\epsilon$ in \eqref{eq:trucmachin} is strictly negative for $\epsilon$ small enough, what we assume from now. Therefore, $\sum_j \mathscr{P}( \abs{M_{j,p}- \mathbb{E}[M_{j,p}]}\geq 2^{-j\epsilon} m_p / 2 ) < \infty$. From the Borel-Cantelli lemma, only a finite number of such $j$ can therefore satisfy the relation $\abs{M_{j,p}- \mathbb{E}[M_{j,p}]}\geq 2^{-j\epsilon} m_p / 2$. A   consequence is then that there exists almost surely a random $J \in \N$ such that for every $j \geq J$, 
		\begin{equation}
			M_{j,p} \geq \mathbb{E} [M_{j,p}] - \abs{M_{j,p} - \mathbb{E} [M_{j,p}] } \geq m_p 2^{-j\epsilon} - \frac{m_p}{2} 2^{-j\epsilon} = \frac{m_p}{2} 2^{-j \epsilon},
		\end{equation} 
		where we used the lower bound in \eqref{eq:boundtoprove} with $q = p$.
	We deduce that
\begin{equation*}
	\lVert w \rVert_{B_p^\tau(\R^d;\rho)}^p 
	\geq 
	C \sum_{j\geq J} 2^{j p (\tau - d + d / \indlocinf)} M_{j,p} 
	\geq 
	\frac{C m_p}{2} \sum_{j\geq J} 2^{j p (\tau - d + d / \indlocinf- \epsilon)}.
	\end{equation*}		
	For $\epsilon$ small enough, we have that $(\tau + d - d /\indlocinf - \epsilon) > 0$ and, therefore, that $\lVert w \rVert_{B_p^\tau(\R^d;\rho)}^p  = \infty$ almost surely. 	\\
	

	\textbf{ If $\rate \geq  - d /p$.} This case has been treated in full generality in Proposition \ref{prop:uneptiteprop}.
	\end{proof}

\section{L\'evy White Noise: the General Case} \label{sec:general}

This section gives us the opportunity to consolidate the results and to deduce the general case from the previous ones. 
We say that a L\'evy white noise $w$ is \emph{non-Gaussian} if its L\'evy measure is not identically zero. In particular, $w$ can have a Gaussian part in the L\'evy-It\^o decomposition (see Proposition \ref{prop:LIdecompo}). 
Proposition~\ref{prop:sparsenoise} characterizes the Besov regularity of non-Gaussian L\'evy white noises, the Gaussian white noise having already been treated in Section \ref{sec:Gaussian}.
We conclude this section with the proof of Theorem \ref{maintheo:regdecaylimit}.

\begin{proposition} \label{prop:sparsenoise}
Fix $0 < p \leq \infty$ and $\tau , \rho \in \R$. 
	Consider a non-Gaussian L\'evy white noise $w$  with Blumenthal-Getoor and moment indices $0\leq \indlocinf \leq \indlocsup \leq 2$ and $0 < \indmom \leq \infty$.
		Then, $w$ is
		\begin{itemize}
		\item  almost surely in $B_p^\reg(\R^d;\rate)$ 
		if $\reg < d/\max(p,\indlocsup) - d$ 
		and $\rate < - d /\min(p,\indmom)$, for $0<p\leq 2$, $p$ an even integer, or $p=\infty$; and 
		\item almost surely not in $B_p^\reg(\R^d;\rate)$
		if $\reg > d/\max(p,\indlocinf) - d$ 
		or $\rate > - d /\min(p,\indmom)$ for every $0 < p \leq \infty$.
		\end{itemize}
\end{proposition}

Proposition \ref{prop:sparsenoise} reveals new information on the local smoothness and the asymptotic growth rate of non-Gaussian L\'evy white noises.
	\begin{corollary} \label{coro:sparsenoise}
	Let $0< p \leq \infty$ and $w$ be a L\'evy white noise with finite moments and Blumenthal-Getoor indices $0 \leq \indlocinf \leq \indlocsup \leq 2$. Then, we have that
	\begin{equation} \label{eq:taurhofinitemom}
	\frac{d}{\max(p,\indlocsup)} - d  \leq \tau_p(w) \leq \frac{d}{\max(p,\indlocinf)} - d  \quad \text{ and } \quad \rho_p(w) = - \frac{d}{\min (p,\indmom)}.
	\end{equation}
	\end{corollary}
	
\begin{proof}[Proof of Proposition \ref{prop:sparsenoise}]
	According to Proposition \ref{prop:LIdecompo}, $w = w_1 + w_2$ with $w_1$ a compound Poisson white noise, $w_2$ a L\'evy white noise with finite moments (which can include a Gaussian part). 
	Moreover, due to Proposition \ref{prop:decomposition}, we have that
		\begin{align}
		\indlocsup 	&= \indlocsup(w_2) \geq \indlocsup(w_1) = 0, \text{ and} \\
		\indmom 	&= \indmom (w_1) \leq \indmom(w_2) =   \infty.
	\end{align}
	
	\textbf{Case  $\reg < \left( \frac{d}{\max(p,\indlocsup)} - d \right) $ and $\rate < - \frac{d}{\min(p,\indmom)}$. }
	Then, $\reg < \left( \frac{d}{p} - d \right) $ and $\rate < - \frac{d}{\min(p,\indmom(w_1))}$ so that $w_1 \in B_{p}^{\reg}(\R^d;\rate)$ due to Proposition \ref{prop:PoissonBesov}. Similarly, $w_2 \in B_{p}^{\reg}(\R^d;\rate)$ due to Proposition \ref{prop:finiteMomentsBesov}. Finally, $w = w_1 + w_2 \in B_{p}^{\reg}(\R^d;\rate)$. \\
	
\textbf{Case  $\reg > \left( \frac{d}{\max(p,\indlocsup)} - d \right) $.} The arguments of Proposition \ref{prop:finiteMomentsBesov} for this case are still valid for $w$.	\\	
	
	\textbf{Case  $\rate >  - \frac{d}{\min(p,\indmom)}$. }
	The case $\rate \geq - d / p$ has been treated in Proposition \ref{prop:uneptiteprop}.
	We therefore already know that $w \notin B_p^\reg(\R^d;\rate)$ if $\tau > d / \max(p,\indlocsup) - d$ or if $\rate \geq - d /p$. The only remaining case is when $p  > \indmom$, $ \tau < d / \max(p,\indlocsup) - d$, and $- d / \indmom < \rate < - d / p$. In this case, $w_2 \in B_p^\reg(\R^d;\rate)$ from Proposition \ref{prop:finiteMomentsBesov}, while $w_1 \notin B_p^\reg(\R^d;\rate)$ with Proposition \ref{prop:PoissonBesov} due to the condition $\rate > - d / \indmom$. Finally, $w \notin B_p^\reg(\R^d;\rate)$ a.s. as the sum of an element a.s. in $B_p^\reg(\R^d;\rate)$ and an element a.s. not in $B_p^\reg(\R^d;\rate)$.
	\end{proof}

Finally, we can translate our results in terms of the local smoothness $\tau_p(w)$ and the asymptotic growth rate $\rho_p(w)$ of L\'evy white noises.

\begin{proof}[Proof of Theorem \ref{maintheo:regdecaylimit}]
	The values of $\tau_p(w)$ and $\rho_p(w)$ are directly deduced from Propositions \ref{prop:GaussBesov}, \ref{prop:PoissonBesov}, and  \ref{prop:sparsenoise}.  
	Positive results ($w \in B_p^{\tau}(\R^d;\rho)$) directly give lower bounds for $\tau_p(w)$ and $\rho_p(w)$, while negative results  ($w \notin B_p^{\tau}(\R^d;\rho)$) provide upper bounds. 
	For the Gaussian and compound Poisson cases, studied separatly, there are no restrictions on $0<p \leq \infty$. The case of a general non-Gaussian L\'evy white noise is deduced from Proposition \ref{prop:sparsenoise}.
\end{proof}

\section{Discussion and Examples} \label{sec:blabla}

	\subsection{Application to  Subfamilies of L\'evy White Noises} \label{subsec:examples}
	
		We apply Theorem \ref{maintheo:regdecaylimit} to deduce the local smoothness and asymptotic growth rate of specific L\'evy white noises. We consider Gaussian, symmetric-$\alpha$-stable~\cite{Taqqu1994stable}, symmetric Gamma (including Laplace)~\cite{Koltz2001laplace}, compound Poisson, inverse Gaussian~\cite{Barndorff1997processes}, and layered stable white noises~\cite{Houdre2007layered}. All the underlying laws are known to be infinitely divisible \cite{Houdre2007layered,Sato1994levy}.
		In Table \ref{table:noises}, we define the different families in terms of the characteristic function of $X= \langle w , \One_{[0,1]^d} \rangle$ and give adequate references.
		Most of these families together with the convention we are following in this paper are introduced and detailed in \cite[Section 5.1]{fageot2019scaling} and~\cite[Section 2.1.3]{Fageotthese}. 
		
		The layered stable white noises have the particularity of describing the complete spectrum of possible couples  $(\alpha_1,\alpha_2) = (\indlocsup,\indmom) \in (0,2)^2$. The characteristic exponent of a layered stable white noise is
		\begin{equation}\label{eq:layeredpsi}
		\Psi_{\alpha_1,\alpha_2} (\xi) = \int_{\R} (\cos(t \xi) - 1 ) \left( \One_{\abs{t}\leq1} \abs{t}^{-(\alpha_1+1)} + \One_{\abs{t}>1} \abs{t}^{-(\alpha_2+1)} \right) \mathrm{d} t.
		\end{equation}	
	
We also provide a visualization of our results in terms of Triebel diagrams. In Figures \ref{fig:taurhogauss} to \ref{fig:taurhogenerallevy}, we plot the local smoothness $\frac{1}{p} \mapsto \tau_p(w)$ and asymptotic growth rate $\frac{1}{p} \mapsto \rho_p(w)$ for different L\'evy white noises (with the exception of $\tau_p(w)$ which is not fully determined for the general case in Figure \ref{fig:taurhogenerallevy}; here, we represent the lower and upper bounds of \eqref{eq:LNtaurho}). A given noise is almost surely in a Besov space $B_p^\reg(\R^d;\rho)$ if the points $(1/p,\reg)$ and $(1/p,\rate)$ are in the lower shaded green regions. \emph{A contrario}, the L\'evy white noise is almost surely not in $B_p^\reg(\R^d;\rho)$ if $(1/p,\reg)$ or $(1/p,\rate)$ are in the upper shaded red region. In Figure~\ref{fig:taurhogenerallevy}, the white region corresponds to the case where we do not know if the L\'evy white noise is or is not in the corresponding Besov spaces, a situation that is examplified in Section  \ref{subsec:handcrafted} and discussed in Section \ref{sec:conclusiveremarks}.
	{In this diagrams, we moreover assume that our lower bound \eqref{eq:LNtaubelow} is valid for any $p>0$, including no even integers when $p\geq 2$. This conjectural point is discussed in Section \ref{sec:conclusiveremarks}.}

\begin{table*}[t!] 
\centering
\caption{L\'evy White Noises and their Indices}

\begin{tabular}{cccccc} 
\hline
\hline\\[-2ex] 
White noise & Parameters &  $\CF_X(\xi)$  & $\indlocsup = \indlocinf$ & $\indmom$   \\
\\
[-2ex] 
\hline\\[-1.5ex]
Gaussian    & $\sigma^2>0$ & $\mathrm{e}^{- \sigma^2 \omega^2 /2}$ & $2$ & $\infty$  \\[+1ex]
Cauchy   \cite{Taqqu1994stable}  & $\gamma > 0$ & $\mathrm{e}^{- \gamma \lvert \xi \rvert}$    & $1$ & $1$   \\[+1ex]
S$\alpha$S  \cite{Taqqu1994stable}  & $0<\alpha<2$ & $\mathrm{e}^{-\lvert \xi \rvert^\alpha}$    & $\alpha$ & $\alpha$   \\[+1ex]
sum of S$\alpha$S	& $0 < \alpha_1 , \alpha_2 \leq 2$ & $\mathrm{e}^{-\lvert \xi \rvert^{\alpha_1} - \lvert \xi \rvert^{\alpha_2}}$ & $\max(\alpha_1, \alpha_2)$ & $\min(\alpha_1, \alpha_2)$   \\[+1ex]
Laplace \cite{Koltz2001laplace}  & $\sigma^2 > 0$ & $(1 + \sigma^2\xi/2)^{-1}$ & $0$ & $\infty$ \\[+1ex]
symmetric Gamma  \cite{Koltz2001laplace} & $\sigma^2 , \lambda > 0$ & $(1 + \sigma^2\xi/2)^{-\lambda}$ & $0$ & $\infty$ \\[+1ex]
compound Poisson & $\lambda >0$ & $\mathrm{e}^{\lambda (\widehat{{P}} (\xi) -1 )}$ & $0$ & $\infty$ \\
with finite moments  & ${P}$ & &  \\[+1ex]
layered stable \cite{Houdre2007layered} & $0 < \alpha_1 , \alpha_2 < 2$ &    see \eqref{eq:layeredpsi} & $\alpha_1$ & $\alpha_2$ \\[+1ex]
inverse Gaussian \cite{Barndorff1997processes} &  -  & $\mathrm{e}^{1 - (1 - 2 \mathrm{i} \xi)^{1/2}}$ & $1/2$ & $\infty$  \\[+1ex]
with finite moment & $(\mu,\sigma^2,\nu)$ & $\mathrm{e}^{\psi(\xi)}$ with $\psi$  & $2$ & $\infty$ \\
and Gaussian part  &   & given by \eqref{eq:LK} &   &\\
\hline
\hline
\end{tabular} \label{table:noises}
\end{table*}

\begin{figure}[h!] 
\centering
\begin{subfigure}[b]{0.30\textwidth}
\begin{tikzpicture}[x=2cm,y=2cm,scale=0.42]

\fill[green, opacity= 0.3] (0,-0.75) -- (3,-0.75) -- (3,-2)  -- (0,-2) -- cycle; 
\fill[red, opacity= 0.3] (0,-0.75) -- (3,-0.75) -- (3,1.5) -- (0,1.5) -- cycle; 
\draw[thick, ->] (0,0)--(3,0) node[circle,right] {$\frac{1}{p}$} ;
\draw[thick, ->] (0,-2)--(0,1.5) node[circle,above] {$\reg$} ;
\draw[ thick,color=black] (-0.05,0) -- (0.05,0)  node[black,left] { $0$};
\draw[ thick,color=black] (-0.05,-1.5) -- (0.05,-1.5)  node[black,left] { $-d$};
\draw[ thick,color=black](1.5,-0.05) -- (1.5,0.1)  node[black,above] { $1$};
\draw[black, thick, ->](0,-0.75) --(3,-0.75);
\draw[ thick,color=black] (-0.05,-0.75) -- (0.05,-0.75)  node[black,left] { $-\frac{d}{2}$};

\end{tikzpicture}
\end{subfigure}
\begin{subfigure}[b]{0.30\textwidth}
\begin{tikzpicture}[x=2cm,y=2cm,scale=0.42]

\fill[green, opacity= 0.3] (0,1) -- (3,-2)  -- (0,-2) -- cycle; 
\fill[red, opacity= 0.3] (0,1) -- (3,-2) -- (3,1.5) -- (0,1.5) -- cycle; 
\draw[thick, ->] (0,1)--(3,1) node[circle,right] {$\frac{1}{p}$} ;
\draw[thick, ->] (0,-2)--(0,1.5) node[circle,above] {$\rate$} ;
\draw[ thick,color=black] (-0.05,1) -- (0.05,1)  node[black,left] { $0$};
\draw[ thick,color=black] (-0.05,-0.5) -- (0.05,-0.5)  node[black,left] { $-d$};
\draw[ thick,color=black](1.5,0.9) -- (1.5,1.1)  node[black,above] { $1$};
\draw[black, thick, ->](0,1) --(3,-2);

\end{tikzpicture}
\end{subfigure}
\caption{Gaussian white noise} \label{fig:taurhogauss}

\begin{subfigure}[b]{0.30\textwidth}
\begin{tikzpicture}[x=2cm,y=2cm,scale=0.42]

\fill[green, opacity= 0.3] (0,-1.5) -- (2.25,0.75) -- (3,0.75)  -- (3,-2) --  (0,-2) -- cycle; 
\fill[red, opacity= 0.3] (0,-1.5) -- (2.25,0.75) -- (3,0.75) -- (3,1.5) -- (0,1.5) -- cycle; 
\draw[thick, ->] (0,0)--(3,0) node[circle,right] {$\frac{1}{p}$} ;
\draw[thick, ->] (0,-2)--(0,1.5) node[circle,above] {$\reg$} ;
\draw[ thick,color=black] (-0.05,0) -- (0.05,0)  node[black,left] { $0$};
\draw[ thick,color=black] (-0.05,-1.5) -- (0.05,-1.5)  node[black,left] { $-d$};
\draw[ thick,color=black](1.5,-0.05) -- (1.5,0.1)  node[black,above] { $1$};
\draw[ thick,color=black](2.25,-0.05) -- (2.25,0.1)  node[black,below] { \footnotesize{$1/ \alpha$}};
\draw[black, thick, -](0,-1.5) --(2.25,0.75);
\draw[black, thick, ->](2.25,0.75) --(3,0.75);

\end{tikzpicture}
\end{subfigure}
\begin{subfigure}[b]{0.30\textwidth}
\begin{tikzpicture}[x=2cm,y=2cm,scale=0.42]

\fill[green, opacity= 0.3] (0,-1.25) -- (2.25,-1.25) -- (3,-2)  -- (0,-2) -- cycle; 
\fill[red, opacity= 0.3] (0,-1.25) -- (2.25,-1.25) -- (3,-2) -- (3,1.5) -- (0,1.5) -- cycle; 
\draw[thick, ->] (0,1)--(3,1) node[circle,right] {$\frac{1}{p}$} ;
\draw[thick, ->] (0,-2)--(0,1.5) node[circle,above] {$\rate$} ;
\draw[ thick,color=black] (-0.05,1) -- (0.05,1)  node[black,left] { $0$};
\draw[ thick,color=black] (-0.05,-1.25) -- (0.05,-1.25)  node[black,left] { $-\frac{d}{\alpha}$};
\draw[ thick,color=black](2.25,0.9) -- (2.25,1.1)  node[black,above] { \footnotesize{${1/\alpha}$}};
\draw[ thick,color=black] (-0.05,-0.5) -- (0.05,-0.5)  node[black,left] { $-d$};
\draw[ thick,color=black](1.5,0.9) -- (1.5,1.1)  node[black,above] { $1$};
\draw[black, thick, -](0,-1.25) --(2.25,-1.25);
\draw[black, thick, ->](2.25,-1.25) --(3,-2);

\end{tikzpicture}
\end{subfigure}
\caption{S$\alpha$S white noise with $\alpha = 2/3$}

\begin{subfigure}[b]{0.30\textwidth}
\begin{tikzpicture}[x=2cm,y=2cm,scale=0.42]

\fill[green, opacity= 0.3] (0,-1.5) -- (3,1.5)  -- (3,-2) --  (0,-2) -- cycle; 
\fill[red, opacity= 0.3] (0,-1.5)-- (3,1.5) -- (0,1.5) -- cycle; 
\draw[thick, ->] (0,0)--(3,0) node[circle,right] {$\frac{1}{p}$} ;
\draw[thick, ->] (0,-2)--(0,1.5) node[circle,above] {$\reg$} ;
\draw[ thick,color=black] (-0.05,0) -- (0.05,0)  node[black,left] { $0$};
\draw[ thick,color=black] (-0.05,-1.5) -- (0.05,-1.5)  node[black,left] { $-d$};
\draw[ thick,color=black](1.5,-0.05) -- (1.5,0.1)  node[black,above] { $1$};
\draw[black, thick, ->](0,-1.5)--(3,1.5);

\end{tikzpicture}
\end{subfigure}
\begin{subfigure}[b]{0.30\textwidth}
\begin{tikzpicture}[x=2cm,y=2cm,scale=0.42]

\fill[green, opacity= 0.3] (0,0.30) -- (0.75,0.30)-- (3,-2)  -- (0,-2) -- cycle; 
\fill[red, opacity= 0.3](0,0.30) -- (0.75,0.30)-- (3,-2) -- (3,1.5) -- (0,1.5) -- cycle; 
\draw[thick, ->] (0,1)--(3,1) node[circle,right] {$\frac{1}{p}$} ;
\draw[thick, ->] (0,-2)--(0,1.5) node[circle,above] {$\rate$} ;
\draw[ thick,color=black] (-0.05,1) -- (0.05,1)  node[black,left] { $0$};
\draw[ thick,color=black] (-0.05,0.30) -- (0.05,0.30)  node[black,left] { $-\frac{d}{\indmom}$};
\draw[ thick,color=black](0.75,0.9) -- (0.75,1.1)  node[black,above] { \footnotesize{$1/\indmom$}};
\draw[ thick,color=black] (-0.05,-0.5) -- (0.05,-0.5)  node[black,left] { $-d$};
\draw[ thick,color=black](1.5,0.9) -- (1.5,1.1)  node[black,above] { $1$};
\draw[black, thick, -](0,0.30) --(0.75,0.30);
\draw[black, thick, ->](0.75,0.30) --(3,-2);

\end{tikzpicture}
\end{subfigure}
\caption{compound Poisson white noise with $\indmom = 2$}
\end{figure}

\begin{figure}[h!] 
\centering
\begin{subfigure}[b]{0.30\textwidth}
\begin{tikzpicture}[x=2cm,y=2cm,scale=0.42]

\fill[green, opacity= 0.3] (0,-1.5) -- (3,1.5)  -- (3,-2) --  (0,-2) -- cycle; 
\fill[red, opacity= 0.3] (0,-1.5)-- (3,1.5) -- (0,1.5) -- cycle; 
\draw[thick, ->] (0,0)--(3,0) node[circle,right] {$\frac{1}{p}$} ;
\draw[thick, ->] (0,-2)--(0,1.5) node[circle,above] {$\reg$} ;
\draw[ thick,color=black] (-0.05,0) -- (0.05,0)  node[black,left] { $0$};
\draw[ thick,color=black] (-0.05,-1.5) -- (0.05,-1.5)  node[black,left] { $-d$};
\draw[ thick,color=black](1.5,-0.05) -- (1.5,0.1)  node[black,above] { $1$};
\draw[black, thick, ->](0,-1.5)--(3,1.5);

\end{tikzpicture}
\end{subfigure}
\begin{subfigure}[b]{0.30\textwidth}
\begin{tikzpicture}[x=2cm,y=2cm,scale=0.42]

\fill[green, opacity= 0.3] (0,1) -- (3,-2)  -- (0,-2) -- cycle; 
\fill[red, opacity= 0.3] (0,1) -- (3,-2) -- (3,1.5) -- (0,1.5) -- cycle; 
\draw[thick, ->] (0,1)--(3,1) node[circle,right] {$\frac{1}{p}$} ;
\draw[thick, ->] (0,-2)--(0,1.5) node[circle,above] {$\rate$} ;
\draw[ thick,color=black] (-0.05,1) -- (0.05,1)  node[black,left] { $0$};
\draw[ thick,color=black] (-0.05,-0.5) -- (0.05,-0.5)  node[black,left] { $-d$};
\draw[ thick,color=black](1.5,0.9) -- (1.5,1.1)  node[black,above] { $1$};
\draw[black, thick, ->](0,1) --(3,-2);

\end{tikzpicture}
\end{subfigure}
\caption{Symmetric-Gamma white noise}

\begin{subfigure}[b]{0.30\textwidth}
\begin{tikzpicture}[x=2cm,y=2cm,scale=0.42]

\fill[green, opacity= 0.3] (0,-1.5) -- (0.75,-0.75) -- (3,-0.75)  -- (3,-2) --  (0,-2) -- cycle; 
\fill[red, opacity= 0.3] (0,-1.5) -- (2.25,0.75) -- (3,0.75) -- (3,1.5) -- (0,1.5) -- cycle; 
\draw[thick, ->] (0,0)--(3,0) node[circle,right] {$\frac{1}{p}$} ;
\draw[thick, ->] (0,-2)--(0,1.5) node[circle,above] {$\reg$} ;
\draw[ thick,color=black] (-0.05,0) -- (0.05,0)  node[black,left] { $0$};
\draw[ thick,color=black] (-0.05,-1.5) -- (0.05,-1.5)  node[black,left] { $-d$};
\draw[ thick,color=black](1.5,-0.05) -- (1.5,0.1)  node[black,above] { $1$};
\draw[ thick,color=black](2.25,-0.05) -- (2.25,0.1)  node[black,below] { \footnotesize{$1/\underline{\beta}_\infty$}};
\draw[ thick,color=black](0.75,-0.05) -- (0.75,0.1)  node[black,above] { \footnotesize{$1 / \indlocsup$}};
\draw[black, thick, -](0,-1.5) --(2.25,0.75);
\draw[black, thick, ->](2.25,0.75) --(3,0.75);
\draw[black,  thick, ->](0.75,-0.75) --(3,-0.75);
\end{tikzpicture}
\end{subfigure}
\begin{subfigure}[b]{0.30\textwidth}
\begin{tikzpicture}[x=2cm,y=2cm,scale=0.42]

\fill[green, opacity= 0.3] (0,-1.25) -- (2.25,-1.25) -- (3,-2)  -- (0,-2) -- cycle; 
\fill[red, opacity= 0.3] (0,-1.25) -- (2.25,-1.25) -- (3,-2) -- (3,1.5) -- (0,1.5) -- cycle; 
\draw[thick, ->] (0,1)--(3,1) node[circle,right] {$\frac{1}{p}$} ;
\draw[thick, ->] (0,-2)--(0,1.5) node[circle,above] {$\rate$} ;
\draw[ thick,color=black] (-0.05,1) -- (0.05,1)  node[black,left] { $0$};
\draw[ thick,color=black] (-0.05,-1.25) -- (0.05,-1.25)  node[black,left] { $-\frac{d}{\indmom}$};
\draw[ thick,color=black](2.25,0.9) -- (2.25,1.1)  node[black,above] { \footnotesize{$1/\indmom$}};
\draw[ thick,color=black] (-0.05,-0.5) -- (0.05,-0.5)  node[black,left] { $-d$};
\draw[ thick,color=black](1.5,0.9) -- (1.5,1.1)  node[black,above] { $1$};
\draw[black, thick, -](0,-1.25) --(2.25,-1.25);
\draw[black, thick, ->](2.25,-1.25) --(3,-2);

\end{tikzpicture}
\end{subfigure}
\label{fig:taurhogenerallevy}
\caption{L\'evy white noise with $ 2/3 = \underline{\beta}_\infty < \indlocsup = 2$ and $\indmom  = 2/3$}
\end{figure}

	\subsection{L\'evy White Noises with Distinct Blumenthal-Getoor Indices}
	\label{subsec:handcrafted} 
	
	In Theorem \ref{maintheo:regdecaylimit}, we obtained lower and upper bounds for the local regularity of a L\'evy white noise (see \eqref{eq:LNtaurho}). This bounds are equal if and only if $\indlocinf = \indlocsup$. This equality is valid for all the examples presented in Section \ref{subsec:examples}.  It is however possible to construct characteristic exponents with indices that take any values $(\indlocinf ,\indlocsup) \in [0,2]^2$ with the obvious constraint that $\indlocinf \leq \indlocsup$. This is done in \cite[Examples 1.1.14, 1.1.15]{Farkas2001function}, where the authors introduce
	\begin{align}
		\Psi_{0,\beta_2,M}(\xi) &=    \sum_{k \geq 1} 2^{\beta_2 M^k - k} (\cos( 2^{-M^k}\xi) - 1), \\
		\Psi_{\beta_1,\beta_2,M}(\xi) &= \int_{\abs{t}\leq1} \frac{\cos(t \xi) - 1}{\abs{t}^{\beta_1+1}} \mathrm{d}t + \Psi_{\beta_2,M}(\xi),
	\end{align}
	with $0<\beta_1 \leq \beta_2 <2$ and $M > 2 / (2-\beta_2)$, and show that $\Psi_{0,\beta_2,M}$ ($\Psi_{\beta_1,\beta_2,M}$, respectively) is a characteristic exponent with Blumenthal-Getoor indices $\indlocinf = 0$ and $\indlocsup = \beta_2$ ($\indlocinf = \beta_1$ and $\indlocsup = \beta_2$, respectively).

	\subsection{Conclusive Remarks and Open Questions} \label{sec:conclusiveremarks}

	We have obtained new results on the localization of L\'evy white noises in weighted Besov spaces, summarized in Theorem \ref{maintheo:regdecaylimit}. This includes the identification of the local smoothness in many cases (including the examples presented in Section \ref{subsec:examples}), and lower and upper bounds for the general case.
	We also identify the asymptotic growth rate of  the L\'evy white noise in many situations, significantly improving known results.
	However, some   questions remain open for a definitive answer regarding the Besov regularity of L\'evy white noise. 
	
	\begin{itemize}
	
	\item { \emph{Moment estimates: the case $p>2$.}
		Our growth results in Theorem \ref{maintheo:regdecaylimit} present some restriction on the integrability parameter $p$, since  relation \eqref{eq:LNtaubelow} presently excludes the case  $p>2$, $p\notin 2\N$.
		{Our derivation makes extensive use of  \eqref{eq:pthmomentRV}, which is only valid for $p<2$. 
		As is classic in moment estimation, the case $p>2$ adds tehnical difficulties and deserve a specific treatment.}
			Nevertheless, we conjecture that the derived formulas remains true for any $0<p \leq \infty$, \emph{i.e.}, that our lower bound \eqref{eq:LNtaubelow2} is sharp for any integrability parameter.}

	\item { \emph{Blumenthal-Getoor indices and the local smoothness.} 
		When $\indlocinf < \indlocsup$ and $p \leq \indlocsup$, we have distinct lower and upper bounds for the local smoothness in Theorem \ref{maintheo:regdecaylimit}. 
		The identification of $\reg_p(w)$ for these cases is unknown at this stage.
		It is actually not even clear if $\reg_p(w)$ can be expressed in terms of the indices considered so far. We believe
		that a precise answer to  this question requires the development of   new tools to capture the precise behavior of the moments in relation with the scale $j$. A first step in this direction will be to consider the examples presented in Section \ref{subsec:handcrafted}.}

		\item \emph{Critical values.}
		We did not investigate the localization of a general L\'evy white noise for the critical values $\tau = \reg_p(w)$ or $\rho = \rate_p(w)$. However, partial answers have been given for compound Poisson white noises (Proposition \ref{prop:PoissonBesov}) and finite-moment white noises (Proposition \ref{prop:finiteMomentsBesov}). A complete characterization was given in the Gaussian case (Proposition \ref{prop:GaussBesov}). For the general case, we conjecture that $w \notin B_{p}^{\tau}(\R^d;\rho)$ as soon as $\tau = \reg_p(w)$ or $\rho = \rate_p(w)$, in accordance with  known results.
	\end{itemize}


\begin{footnotesize}
\bibliographystyle{amsplain}
\bibliography{references}
\end{footnotesize}




\ACKNO{This work has benefited from exchanges with Arash Amini, Carsten Chong, Robert Dalang, Felix Hummel, St\'ephane Jaffard, Alireza Fallah, Ren\'e Schilling, Philippe Th\'evenaz, Michael Unser, Virginie Uhlman, and John Paul Ward. 
The research leading to these results has received funding from the European Research Council under Grant H2020-ERC (ERC grant agreement $\text{n}^\circ$ 692726-GloblBioIm) and the Swiss National Science Foundation with grant agreement P2ELP2\_181759.}


\end{document}